\documentclass[11pt]{amsart}
\usepackage{graphicx}
\usepackage{amssymb,amscd,amsthm,amsxtra}
\usepackage{latexsym}
\usepackage{epsfig}

\usepackage[usenames,dvipsnames,svgnames,table]{xcolor}
\usepackage[pagebackref,linktocpage=true,colorlinks=true,linkcolor=Blue,citecolor=BrickRed,urlcolor=RoyalBlue]{hyperref}
\usepackage[alphabetic]{amsrefs}
\usepackage{mathrsfs}  

\usepackage[colorinlistoftodos,prependcaption,textsize=tiny]{todonotes}

\usepackage[mathscr]{euscript} 

\usepackage{pxfonts}

\newtheorem{thm}{Theorem}[section]

\newtheorem{lem}[thm]{Lemma}
\newtheorem{prop}[thm]{Proposition}
\newtheorem{defn}[thm]{Definition}
\newtheorem{rem}[thm]{Remark}
\numberwithin{equation}{section}

\newcommand{\comment}[1]{}
\newcommand\eqdef{\stackrel{\text {def}}{=}}

\newcommand{\deriv}[2]{\frac{\partial #1}{\partial #2}}

\newcommand{\pscal}[2]{\left\langle #1,#2\right\rangle}

\newcommand{\bblue}[1]{\color{black} #1}

\newcommand{\Cinf}{C^\infty}

\newcommand\pl{\Delta_p}
\newcommand\ue{u^\e}

\newcommand \uej{u^{\e_j}}
\newcommand \uejk{u^{\e_{j_k}}}

\newcommand\bet[1]{\beta_{\e}(#1)}

\newcommand\R{\mathbb R} 
\newcommand\N{\mathbb N} 

\renewcommand\div{\mbox{div}}

\newcommand{\supp}{{ \rm supp}}
\newcommand{\dist}{{ \rm dist}}
\newcommand\p{\partial}

\newcommand\e{\varepsilon}
\renewcommand{\epsilon}{\e}

\renewcommand\H{\mathscr H}

\newcommand{\I}[1]{\chi_{\{#1>0\}}}
\newcommand{\po}[1]{\{#1>0\}}
\newcommand{\fb}[1]{\partial\{#1>0\}}

\newcommand{\vol}{{\rm{ Vol}}}

\newcommand\Om{\Omega}
\newcommand\na{\nabla}

\usepackage{tikz}
\usetikzlibrary{arrows,shapes,positioning}
\usetikzlibrary{decorations.markings}
\tikzstyle arrowstyle=[scale=1]
\tikzstyle directed=[postaction={decorate,decoration={markings,
    mark=at position .65 with {\arrow[arrowstyle]{stealth}}}}]
\tikzstyle reverse directed=[postaction={decorate,decoration={markings,
    mark=at position .65 with {\arrowreversed[arrowstyle]{stealth};}}}]

\usetikzlibrary{shadows}

\tikzset{
diagonal fill/.style 2 args={fill=#2, path picture={
\fill[#1, sharp corners] (path picture bounding box.south west) -|
                         (path picture bounding box.north east) -- cycle;}},
reversed diagonal fill/.style 2 args={fill=#2, path picture={
\fill[#1, sharp corners] (path picture bounding box.north west) |- 
                         (path picture bounding box.south east) -- cycle;}}
}

\numberwithin{equation}{section} 

\begin{document}

\title[Regularity for two phase problems]
{Regularity for the two phase singular perturbation problems}

\author{Aram Karakhanyan}
\address{School of Mathematics, The University of Edinburgh, Peter Tait Guthrie Road, EH9 3FD
Edinburgh, UK}
\email{aram6k@gmail.com}

\thanks{ 35R35; 35J60}

\begin{abstract}
We prove that an a priori BMO gradient estimate for the two phase singular perturbation problem implies  Lipschitz regularity for the limits.  This problem
 arises in the mathematical theory of combustion, where the  
reaction-diffusion is modelled by the $p$-Laplacian. 
A key tool in our approach is the weak energy identity. 
Our method provides a natural and intrinsic characterization of the free boundary points 
and can be applied to more general classes of solutions. 
 
\end{abstract}

\maketitle

\setcounter{tocdepth}{2}
\tableofcontents
\section{Introduction}
A chief difficulty in working with two phase nonlinear free boundary problems
is the absence of monotonicity formulas. 
One question still unanswered is whether the weak or variational solutions to the two phase free boundary problems have optimal regularity.
These solutions, for instance, arises in the 
mathematical theory of combustion, in the models of high activation of energy.  

In this paper we address this question by developing a unified approach for the nonlinear two phase free boundary problems.
To elucidate our  main ideas we start from the singular perturbation problem for the $p$-Laplacian. 

 Let  $u^\e$ be a family of solutions to 
\begin{equation}\label{pde-0}\tag{$\mathcal P_\e$}
\pl \ue=\bet{\ue}, \quad |\ue|\le 1, \quad \mbox{in}\ B_1, 
\end{equation}
where $\e>0$ is a parameter and 
\begin{equation}\label{eq:beta-prpts}
\bet t=\frac1\e\beta\left(\frac t{\e}\right), \quad 0\le \beta\in\Cinf_0(0,1), \quad \int_{[0, 1]}\beta:=M>0.
\end{equation}
The  quasilinear operator $\pl\ue=\div(|\na \ue|^{p-2}\na\ue), 1<p<\infty$ is called the 
$p$-Laplacian. If $p>2$ then $\Delta_p$ is degenerate elliptic,   and it depicts  diffusion obeying power law.


\subsection{Known results}
For $p=2$, Zel{`}dovich and Frank-Kamenetskii studied the one dimensional version of 
\eqref{pde-0} in 1938, see \cite{Zeld} Chapter 1.4,  and calculated  the speed of 
the front which is $\sqrt{2M}$, see formula (4.30) there. 
In high dimensions the one phase  problem for the Laplacian 
was studied by Berestycki,  Caffarelli and Nirenberg \cite{BCN},  Caffarelli Vazquez \cite{CV}, 
and the 
two phase problem in  \cite{Caff-lip},  \cite{CLW-uniform} for the heat equation 
$\Delta \ue-\deriv {\ue}t=\bet\ue$. Later  
Caffarelli and Kenig studied the two phase problem for the variable coefficient case $\div(a\na\ue )-\deriv{\ue}t=\bet\ue$ \cite{Caff-Kenig}. 
The key tool the authors used in the proof of optimal regularity is the monotonicity formula of Caffarelli \cite{Caff-heat-mon}.
With its help one can establish local uniform Lipschitz  estimates in the parabolic distance.

For the nonlinear operators the extensions of these results are available only for 
the one phase problem, i.e. when $\ue\ge0$, see  
 \cite{Wolan-px}, \cite{Diego}, \cite{Led-quasi}, \cite{Dan}. 

\subsection{Heuristic discussion}
Heuristically, the limits of $u^\e$ as $\e\to 0$ are the solutions of the two phase Bernoulli type   free boundary problem
\begin{equation}
\left\{
\begin{array}{lll}
\Delta_p u=0, \quad &\ \text{in}\ (\po u\cup\{u<0\})\cap B_1,\\
|\na u^+|^p-|\na u^-|^p=pM,  &\ \text{on}\ \fb u.
\end{array}
\right.
\end{equation}

In order to pass to the limit we need some uniform 
bounds for $\ue$, say $\|u^\e\|_{L^\infty(B_1)}\le 1$. 
Then from Caccioppoli's inequality  it follows that $\sup_\e\|u^\e\|_{W^{1, p}(B_R)}\le C(R, n,p)$, for every $R<1$.
Assume that $p>n$,  then from  Sobolev's embedding theorem we infer that $\ue$ is locally uniformly H\"older continuous. 
Consequently, if $u(x)>0$ for some $x\in B_1$, then $u$ is $p$-harmonic in some neighborhood of $x$. 
Moreover, the uniform convergence 
also implies that $\na \uej\to \na u$ strongly in $L^p_{loc}(B_1)$. 

\bigskip 

Using  these observations it is easy to see that 
the limits $u$ of the singular perturbation problem  satisfy the weak energy 
identity 
\[
\int \left[|\na u|^p+p\mathcal B^\ast(x))\right]\div X=p\int|\na u|^{p-2}\na u \na X\na u, 
\]
where {\bblue $X=(X^1, \dots, X^n)$} is a $C^1$ vectorfield with support in $B_1$ {\bblue and $\nabla X=\partial_i X^j$ is the gradient of $X$.}
The function $\mathcal B^\ast$ is bounded and characterizes the concentration of 
the measure $\Delta_p \ue$ on the free boundary $\fb u$ as $\e\to 0$. 
If $\fb u$ is $C^1$ then one can see that $\mathcal B^\ast=M\I u$ with $M=\int \beta$, see Lemma \ref{lem:8}.
In particular,  every stationary point of the functional $\int_{B_1}|\na u|^p+M\I u$ satisfies the weak
energy identity.

We split $\fb u$  into three subsets:

\begin{itemize}
\item[(A)] $x_0\in \fb u$ is a flat point, 
\item[(B)]  the Lebesgue density of $\{u<0\}$ is small at $x_0$, 
\item[(C)]  neither (A) nor (B) hold at $x_0$.  
\end{itemize}

We remark that there may be solutions $u$, obtained as limits of \eqref{pde-0},  of the form $\alpha(x-x_0)_1^++\bar \alpha(x-x_0)_1^-+o(|x-x_0|)$
near $x_0$ with $\alpha, \bar\alpha\ge0$ and at these points the stratification argument \cite{DK} fails.
This is the reason why we further split the flat points into two parts and use the Lebesgue density to identify the 
points where $u$ is a viscosity solution. 

\medskip 

If $u$ fails to have linear growth at some point $x_0\in \fb u $ then the scaled functions 
$$
u_r(x)=\frac{u(x_0+rx)}{\sup\limits_{B_r(x_0)} |u|}
$$
converge to a limit function $u_0$ satisfying the identity  
\begin{equation}\label{fargo}
\int |\na u_0|^p\div X=p\int|\na u_0|^{p-2}\na u_0 \na X\na u_0, \quad X\in C_0^1(\R^n, \R^n). 
\end{equation}
We claim that \eqref{fargo} implies that 
$\Delta_p u_0=0$ in $\R^n$. The converse statement is obviously true since \eqref{fargo} is the domain variation of the energy $\int|\na u|^p$.
This is an interesting question of independent interest since \eqref{fargo}, as we show in this paper,  gives another characterization of the 
$p$-harmonic functions for $p>n$. That done, we can apply  Liouville's theorem to conclude that $u_0$ is a linear function.  
Combining this with the stratification argument from \cite{DK}
with respect to the modulus of continuity of the slab flatness and the Lebesgue 
density of $\{u<0\}$ we conclude that if either (A) or (B) hold then $u$ has linear growth at $x_0$. 
For the remaining case (C) we can conclude that $u$ is a viscosity solution  and $x_0$ is a flat point, so from the Harnack 
principle  we infer that $\fb u$ near $x_0$ is $C^{1, \gamma}$ smooth hypersurface, and the linear growth 
for this case follows from the standard boundary gradient estimates for $u$.

A schematic view of the main steps in the proof of Lipschitz regularity is illustrated  in Figure 1. 

\tikzstyle{decision} = [diamond, draw, fill=red!50, 
    text width=4.5em, text badly centered, node distance=3cm, inner sep=0pt]
\tikzstyle{block} = [rectangle, draw, fill=yellow, 
    text width=12em, text centered, rounded corners, minimum height=4em]
    \tikzstyle{block1} = [rectangle, draw, fill=green!80, 
    text width=9em, text centered, rounded corners, minimum height=4em]
    \tikzstyle{bigblock} = [rectangle, draw, fill=orange!60, 
    text width=7em, text centered, rounded corners, minimum height=4em]
    \tikzstyle{soblock} = [rectangle, draw, fill=cyan!60, 
    text width=6em, text centered, rounded corners, minimum height=4em, drop shadow]
\tikzstyle{line} = [draw, -latex']
\tikzstyle{cloud} = [draw, ellipse,fill=yellow!60, node distance=3cm,
    minimum height=2em]
\tikzstyle{half} = [diagonal fill={cyan!60}{cyan!60},
      text width=9em, minimum height=4em,
      text centered, rounded corners, draw, drop shadow]

\begin{figure}
\small
\begin{tikzpicture}[node distance = 2cm, auto, scale=0.3 ]
      \vspace{0pt};
    \node [block1, node distance = 1.5cm](dic1){Suppose $x_0$ is a free boundary point};
                \node[soblock, below of=dic1, node distance = 3cm] (gen1){$x_0$ is not $h_0$-flat point at scale $r>0$}; 
                    \node[block, below of=gen1, node distance = 3.5cm] (gen2){Linear growth at $x_0$, i.e. $$|u(x)|\le C|x-x_0|, x\in B_{r/2}(x_0)$$}; 
                        \node[half, right of=gen1, node distance = 4cm] (simp2){$x_0$ is $h_0$-flat in $B_r(x_0)$ and $\Theta(u, x_0, r)<\delta$}; 
                           \node[half, right of =simp2, node distance = 4.5cm] (simp1){$x_0$ is $h_0$-flat in $B_r(x_0)$ and $\Theta(u, x_0, r)\ge \delta$}; 
                              \node[bigblock, right of =simp1, node distance = 4cm] (simp3){$u$ is a viscosity solution and $\partial\{ u>0\}$ is  $C^{1, \gamma}$ regular hypersurface}; 
           \path[line] (dic1)  -- (gen1);
                  \path[line] (gen1)-- node {Yes} (gen2);
                           \path[line] (gen1) --node  {No} (simp2);
                                                      \path[line] (simp2) --node  {Yes} (gen2);
                                                                                 \path[line] (simp3) |- node  {} (gen2);
                           \path[line] (simp1) --node{Yes}(simp3);
                           \path[line] (simp2) -- node {No} (simp1);
    \end{tikzpicture}
    \caption{The diagram schematizes the proof of Theorem \ref{thm-A}
    {\bblue which is based on a series of dichotomies depending on 
    the conditions in the blue boxes. The conclusion is in the yellow box, which gives the main result  stated in Theorem \ref{thm-A}.} The constant $C$ depends only on 
    $h_0, \delta, n, p $ and  $M$. }
    \end{figure}
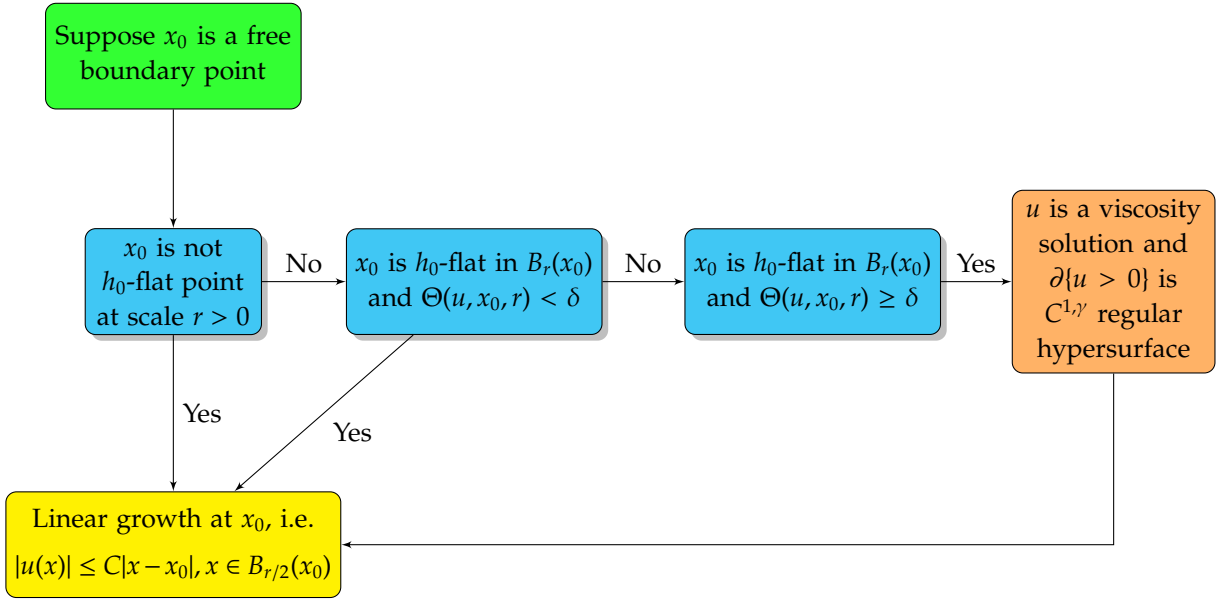

\subsection{Main results}

\begin{thm}\label{thm-A}
Let $\uej$ be a family of solutions to \eqref{pde-0}
such that 
$\uej\to u$ in $W^{1, p}_{loc}$.
If there is a universal constant $C>0$ such that 
\begin{equation}\label{BMO}
\|\na u\|_{BMO_{loc}(B_r(x_0))}\le C(\|u\|_{W^{1, p}(B_{2r}(x_0))}), \quad B_{2r}(x_0)\subset B_1, 
\end{equation}
 then $u$ is locally Lipschitz continuous.  
\end{thm}

The paper is organized as follows: 
Section \ref{sec:tools} contains some well known estimates for 
the solutions and subsolutions of the $p$-Laplacian.

In Section \ref{sec:weak-en} we prove the weak energy identity. 
If the Lipschitz continuity fails then we get that the weak energy identity 
simplifies. In  
Section \ref{sec:visc-space} we show that $BMO$ solutions satisfying this simplified identity must be $p$-harmonic

Then we consider the scenarios (A), (B) and (C) in 
Sections \ref{sec:dyadic}
 (Flatness vs linear growth), 
\ref{sec:dens} (density of negative set) 
and
\ref{sec:viscosity} (viscosity solutions), respectively. 
Combining our results in Section \ref{sec:main-theorem} we give the
proof of Theorem \ref{thm-A}

In Section \ref{sec:basic} we study the properties of solutions with Lipschitz regularity.
Section \ref{sec:weak} is devoted to the weak solutions.
Finally, in  Section \ref{sec:BMO} we prove a BMO type estimate for the tensor $T_{lm}=p|\na u|^{p-2}u_lu_m-|\na u|^p\delta_{lm}$.

\section*{Notation} 
We fix some notation. The $n$ dimensional 
Euclidean space is denoted by $\R^n$,  $u^+=\max(u, 0)$ is the 
nonnegative part of $u$ and similarly 
$u^-=-\min(u, 0)$, so that $u=u^+-u^-$.
The partial derivatives in $x_i$ variable, $i=1, \dots, n$  are denoted by $\p_i u$ 
or $u_i$so that 
 $\p_i u=\deriv u{x_i}$, $(x-x_0)_1$ is the first slot function of the 
 vector $x-x_0$. For every $u\in W^{1, p}(B_1)$ we also define the 
 tensor $T_{lm}(\na u)=p|\na u|^{p-2}u_lu_m-|\na u|^p\delta_{lm}. $
 {If \bblue $X=(X^1, \dots, X^n)$ is a $C^1$ vectorfield, $X:\Omega\to \R^n$ then the tensor $\nabla X=\partial_i X^j$ denotes the gradient of $X$.}
Sometimes we let $\Gamma=\Gamma_u$ to denote  the free boundary $\fb u$ when no confusion can arise, $\vol(E)$ denotes the $n$-dimensional volume of a set $E$. 
\section{Preliminaries and tools: Uniform estimates and compactness}
\label{sec:tools}

\begin{defn}
Let $1<p<\infty$. A function $\ue\in W^{1, p}(B_1)$ is said to be a weak solution to 
$\Delta_p \ue=\beta_\e(\ue)$ in $B_1$ if for every 
$\psi\in W_0^{1, p}(B_1)$ there holds
\[
-\int|\na \ue|^{p-2}\pscal {\na \ue}{\na \psi}=\int\beta_\e(\ue)\psi.
\]
If $\Delta_p u=0$ then $u$ is called $p$-harmonic in $B_1$.
\end{defn}

We recall the well known inequality \cite{DM}
\begin{equation}\label{eq:p-coerc}
\pscal{|\xi|^{p-2}\xi-|\eta|^{p-2}\eta}{\xi-\eta}\ge \gamma|\xi-\eta|^p, p>2.
\end{equation}

\subsection{Caccioppoli inequality and local compactness}
\begin{lem}\label{lem:Cacc}
Let $\ue$ be a family of solutions of \eqref{pde-0}, then exists a constant $C=C(n, p)>0$
depending only on $n, p$ such that for every 
$R\in(0, 1)$ there holds
\begin{equation}\label{eq:Cacc}
\int_{B_R}|\na \ue|^p\le \frac C{(1-R)^p}\int_{B_1}|\ue|^p.
\end{equation}

\end{lem}
\begin{proof}
Let $\eta\in \Cinf_0(B_1), 0\le \eta\le 1, |\na \eta|\le C/R$ for some $C=C(n)>0$, and $\eta=1$ on $B_R$. From  \eqref{eq:beta-prpts}
it follows that $\ue\beta_\e(\ue)\ge 0$. Thus using $\ue\eta^p$ as a test function in the weak
formulation of the equation $\Delta_p\ue=\beta_\e(\ue)$ we get 
\[
0\le \int |\na \ue|^{p-2}\pscal{\na \ue}{\na \ue\eta^p+p\ue \eta^{p-1}\na \eta}.
\] 
Rearranging the terms and using the H\"older inequality we obtain 
\begin{eqnarray*}
\int|\na \ue|^p\eta^p
&\le&
 -p\int|\na \ue |^{p-2}\pscal{\na \ue} {\na \eta}\ue \eta^{p-1}\\
&\le&
p\int|\na \ue |^{p-1}|\na \eta||\ue| \eta^{p-1}\\
&\le&
p\left(\int |\na \ue |^p\eta^p\right)^{1-\frac1p}\left(\int|\ue|^p\eta^p\right)^{\frac1p}.
\end{eqnarray*}
From here we see that 
\[
\int |\na \ue|^p\eta^p \le p^p \int |\ue|^p\eta^p, 
\]
and the desired estimate follows with $C(n, p)=p^p C^p(n)$. 
\end{proof}

In the next Proposition we assume $p>n$,  then from \eqref{eq:Cacc}, the assumption $|\ue|\le 1$,  and 
Sobolev's embedding theorem it follows that 
\[
\|u^\e\|_{C^{1-\frac n{p}}(\overline{B_R})}\le C(R, p, n),
\]
 for fixed  $ 0<R<1$ and uniformly in $\e$.
\begin{prop}\label{prop-comp}
Let $\ue$ be a family of solutions to \eqref{pde-0} and $p>n$. Then the following statements hold true: 
for every sequence $\e_k\to 0$ there is a subsequence, still labelled $\e_k$, and a function $u$ such that 
\begin{itemize}
\item[(i)] $u^{\e_k}\to u$ in $C^{1-\frac np}_{loc}(B_1)\cap W^{1, p}_{loc}(B_1)$,  and $\Delta_p u =0$ in $\po u\cup \{u<0\}$, 
\item[(ii)] $|\na u^{\e_k}|^{p-2}\na u^{\e_k}   
\stackrel{weakly}{\rightharpoonup}
|\na u|^{p-2}\na u$,
\item[(iii)] $|\na u^{\e_k}|^{p}  
\to 
|\na u|^{p}$ strongly in every compact of $B_1$.

\end{itemize}
\end{prop}
\begin{proof}

From 
Lemma \ref{lem:Cacc} and Sobolev's embedding theorem it follows that 
\[
\sup_{k}\left(\|u^{\e_k}\|_{C^{1-\frac n{p}}(\overline{B_R})}+ \|u^{\e_k}\|_{W^{1, p}_{loc}(B_R)}\right)\le C(R, p, n),
\]
Thus (i) follows from a standard compactness argument. 

The proof of (ii) is standard, see \cite{DM}.

To prove (iii) it is enough to  show  that 
$$
\int|\na \ue|^p\psi\to \int|\na u|^p\psi.
$$
This and the weak convergence imply strong convergence.

It follows from \eqref{eq:beta-prpts} that $\ue\beta_\e(\ue)\ge 0$. 
Using $\ue\psi$ as test function in the weak formulation of the equation we get 
as in the proof of Lemma \ref{lem:Cacc}
\begin{eqnarray}
\int|\na \ue |^p\psi\le -\int|\na \ue |^{p-2}\pscal{\na \ue}{\na \psi} \ue.
\end{eqnarray}
Applying Fatou's lemma we obtain
\begin{equation}\label{blya-1}
\int|\na u |^p\psi\le \liminf_{\e\to 0} \int|\na \ue |^p\psi\le -\int|\na u |^{p-2}\pscal{\na u}{\na \psi} u.
\end{equation}

Let $s>0$ be a small number.  Then $(u-s)^+\psi\in W^{1, p}_0(\po u)$. Therefore 
\[
\int_{\{u>s\}}|\na u|^p\psi=-\int|\na u |^{p-2}\pscal{\na u}{\na \psi} (u-s)^+.
\]
Similarly, $(u+s)^-\psi\in W^{1, p}_0(\{ u<0\})$ hence 
\[
\int_{\{u<-s\}}|\na u|^p\psi=-\int|\na u |^{p-2}\pscal{\na u}{\na \psi} (u+s)^-.
\]
After sending $s\to 0$ we conclude 
\[
\int_{}|\na u|^p\psi\ge \int_{\{u\not =0\}}|\na u|^p\psi=-\int|\na u |^{p-2}\pscal{\na u}{\na \psi} u.
\]
This and \eqref{blya-1} imply 
\[
-\int|\na u |^{p-2}\pscal{\na u}{\na \psi} u 
\le 
\int_{}|\na u|^p\psi
\le 
 \liminf_{\e\to 0} \int|\na \ue |^p\psi
 \le 
-\int|\na u |^{p-2}\pscal{\na u}{\na \psi} u.
\]
Consequently, 
\[
\int_{}|\na u|^p\psi
=
 \lim_{\e_j\to 0} \int|\na \uej |^p\psi.
\]
\end{proof}
\begin{rem}\label{xrusch}
The assumption $p>n$ is technical. It allows 
to get the uniform continuity of $\ue$, see also the discussion in Section \ref{sec:BMO}.  
\end{rem}

\begin{lem}
There exists a constant $C=C(n, p)$ depending only on $n, p$ such that 
for every 
$B_{2R}(x)\Subset B_1$ the measure $\mu=\Delta_p \ue$ satisfies the inequality 
\begin{equation}\label{eq:mu-est}
\int_{B_R(x)}d\mu\le C R^{\frac np-1} \left( 
\int_{B_{2R}(x)}|\na \ue|^p
\right)^{1-\frac1p}.
\end{equation}

\end{lem}
\begin{proof}
Using the divergence theorem we get the estimate 
\begin{eqnarray}\label{blya-2}
\int_{B_r(x)}\Delta_p\ue=\int_{\p B_r(x)}|\na \ue|^{p-2}\pscal{\na \ue}\nu\le \int_{\p B_r(x)}|\na \ue|^{p-1}.
\end{eqnarray}
Integrating both sides of \eqref{blya-2} over $r\in[0, R]$ we obtain
\begin{eqnarray}\label{blya-3}
\int_0^R\int_{\p B_r(x)}|\na \ue|^{p-1}=\int_{ B_R(x)}|\na \ue|^{p-1}\le
\left( 
\int_{B_R(x)}|\na \ue|^p
\right)^{1-\frac1p}|B_R|^{\frac1p}.
\end{eqnarray}
On the other hand 
\begin{eqnarray*}
\int_0^R\int_{B_r(x)}d\mu\ge \frac R2\int_{B_{\frac R2}(x)}d\mu.
\end{eqnarray*}
Combining this with \eqref{blya-3} we get 
\[
\int_{B_{\frac R2}}d\mu\le \frac {2|B_R|^{\frac1p}}R\left( 
\int_{B_R(x)}|\na \ue|^p
\right)^{1-\frac1p},
\]
and \eqref{eq:mu-est} follows. 
\end{proof}
\begin{prop}\label{prop-seq-comp}
Let $u_j$ be a sequence of solutions to $\Delta_p u_j=\mu_j$ in $B_1$
such that $\sup_j\|u_j\|_{W^{1, p}(B_1)}<\infty$ and $\mu_j$ are Radon measures in $B_1$
such that $\supp \mu_j\subset \fb {u_j}$.
If $u_0$ is a limit of $u_j$ then $\na u_j\to \na u_0$ strongly in $L^p_{loc}(B_1)$. 
\end{prop}
\begin{proof}
Let $0\le \psi\in C^1_0(B_1)$ then $u_j\psi \Delta_p u_j\ge 0$ in $B_1$.
Thus we have 
\begin{eqnarray}
\int|\na u_j |^p\psi\le -\int|\na u_j |^{p-2}\pscal{\na u_j}{\na \psi} u_j.
\end{eqnarray}
By Fatou's lemma  
\begin{equation}\label{blya-1}
\int|\na u |^p\psi\le \liminf_{j\to \infty} \int|\na u_j |^p\psi\le -\int|\na u |^{p-2}\pscal{\na u_0}{\na \psi} u_0.
\end{equation}

Let $s>0$ be a small number.  Then $(u_0-s)^+\psi\in W^{1, p}_0(\po {u_0})$. Therefore 
\[
\int_{\{u_0>s\}}|\na u_0|^p\psi=-\int|\na u_0 |^{p-2}\pscal{\na u_0}{\na \psi} (u_0-s)^+.
\]
Similarly, $(u_0+s)^-\psi\in W^{1, p}_0(\{u_0<0\})$ hence 
\[
\int_{\{u_0<-s\}}|\na u_0|^p\psi=-\int|\na u_0 |^{p-2}\pscal{\na u_0}{\na \psi} (u_0+s)^-.
\]
After sending $s\to 0$ we conclude 
\[
\int_{}|\na u_0|^p\psi\ge \int_{\{u_0\not =0\}}|\na u_0|^p\psi=-\int|\na u_0 |^{p-2}\pscal{\na u_0}{\na \psi} u_0.
\]
This and \eqref{blya-1} imply 
\[
-\int|\na u_0 |^{p-2}\pscal{\na u_0}{\na \psi} u_0 
\le 
\int_{}|\na u_0|^p\psi
\le 
 \liminf_{j\to \infty} \int|\na u_j |^p\psi
 \le 
-\int|\na u_0 |^{p-2}\pscal{\na u_0}{\na \psi} u_0.
\]
Consequently, 
\[
\int_{}|\na u_0|^p\psi
=
 \lim_{j\to \infty} \int|\na u_j|^p\psi.
\]

\end{proof}


\medskip 
We  summarize the previous results replacing $B_1$ by a general domain $\mathcal D$. 
\begin{prop}\label{prop-compactness}
Let $u^\e$ be a family of solutions to \eqref{pde-0} in a
domain $\mathcal D\subset \R^n$, and $p>n$. Let us assume that
$\|u^\e\|_{L^\infty(\mathcal D)}\le A$ for some constant
$A>0$ independent of $\e.$ For every $\e_k\to 0$ there
exists a subsequence $\e_{k'}\to 0$ and $u\in C^{1-\frac np}_{loc}(\mathcal D)$, such that
\begin{itemize}
\item[(i)] $u^{\e_{k'}}\to u$ uniformly on compact subsets of $\mathcal D$,
\item[(ii)] $\na u^{\e_{k'}}\to \na u$ in $L^p_{loc}(\mathcal D)$,
\item[(iii)] $u$ is $p$-harmonic in $\mathcal D\cap (\po u\cup \{u<0\}$.
\end{itemize}

\end{prop}


\medskip
\subsection{First and second blow-up}
Using  Proposition \ref{prop-compactness} we can extract a sequence  $u^{\e_j}$ for some sequence $\e_j$
such that $u^{\e_j}\to u$ uniformly  in $B_{\frac12}$.
Let $0<\rho_j\downarrow 0$, $x_j\in \fb u$ and set $u_j(x)=\frac{u(x_j+\rho_j x)}{m_j}$, where 
$m_j$ are some positive numbers such that $\sup_j\rho_j/m_j<\infty$. 
Suppose $u_j$ is uniformly bounded and $u_j\to U$ locally uniformly in $\R^n$ for some function $U$
defined on $\R^n$.

The function $U$ is called a blow-up limit of $u$ with respect to 
  free boundary
points $x_j$ and, in general,  it depends on $\{\rho_j\}$ and $x_j$.

The two propositions to follow establish an important property of the blow-up limits, namely that
the first and second blow-ups of $u$ can be obtained from \eqref{pde-0} for a suitable choice of  parameter
$\e$. 

If $\uej$ solves \eqref{pde-0} then the scaled functions $\hat u^{\e_{j}}(x)=\frac{\uejk(x_{j}+\rho_{j} x)}{m_{j}}$
verify
\begin{eqnarray}\nonumber
\Delta_p  (\hat u^{\e_{j_k}})
&=&
\frac{\rho_j^p}{m_j^{p-1}}\frac{1}{\e_j}\beta\left( \frac{u^{\e_j}(x_j+\rho_j x)}{\e_j}  \right)
=
\left[\frac{\rho_j}{m_j}\right]^p
\frac{1}{\e_j/m_j}\beta\left( \frac{\hat u^{\e_j}(x)}{\e_j/m_j}  \right)\\\label{eq-lap-scale}
&=&
\left[\frac{\rho_j}{m_j}\right]^p
\frac{1}{\delta_j}\beta\left( \frac{\hat u^{\e_j}(x)}{\delta_j}  \right), 
\end{eqnarray}
where $\delta_j=\frac{\e_j}{m_j}$.

If  $\delta_j=\frac{\e_j}{m_j}\to 0$ we see that
$ \hat u^{\e_j}$ solves $\Delta_p \hat u^{\e_j}=\left[\frac{\rho_j}{m_j}\right]^p\beta_{\delta_j}(\hat u^{\e_j})$.

\begin{prop}\label{prop-1st-blow}
Let $u^{\e_j}$ be a family of solutions to \eqref{pde-0}
in a domain $\mathcal D\subset \R^n$ such that
$u_{\e_j}\to u$ uniformly on $\mathcal D$ and $\e_j\to 0.$

Let $x_0\in \mathcal D\cap \fb u$ and let $x_k\in \fb u$ be such that
$x_k\to x_0$ as $k\to \infty$. 

Let $\rho_k\to 0, 
u_k(x)=1/m_ku(x_k+\rho_kx)$, $\hat u^{\e_j}_k(x)=1/m_k\uej(x_k+\rho_kx)$,  and $u_k\to U$ uniformly on compact subsets
of $\R^N$. 

Then there exists $j(k)\to \infty$ such that for every $j\ge j(k)$ 
there holds that $\e_{j}/\rho_k\to 0$ and
\begin{itemize}
\item[(i)] $\hat u^{\e_{j}}_k\to U$ uniformly on compact subsets of
$\R^n$,
\item[{(ii)}] $\na \hat u^{\e_{j}}_k\to \na U$ in $L^p_{loc}(\R^n),$
\item[(iii)] $\na u_k\to \na U$ in $L^p_{loc}(\R^n)$.
\end{itemize}
\end{prop}
\begin{proof}
Our proof closely follows that of  Lemma 3.2 in \cite{CLW-uniform} where the case of $\rho_k=m_k$.
We estimate the difference 
\[
\begin{split}
\frac{\uej(x_k+\rho_k x)}{m_k}-U(x)
=
\frac{\uej(x_k+\rho_k x)}{m_k}-\frac{u(x_k+\rho_k x)}{m_k}\\
+\frac{u(x_k+\rho_k x)}{m_k}-U(x)=I+II
\end{split}
\]
Fix $r>0$, then for every $\delta>0$ and $R<\frac r{\rho_k}$
there exists $k_0=k_0(\delta, R)$  such that for $k>k_0$ there holds
\[
|II|=\left|\frac{u(x_k+\rho_k x)}{m_k}-U(x)\right|<\delta, \quad x\in B_R.
\]
Let $x\in B_r(x_0)$, then there is $j(k)$ such that $|\uej(x)-u(x)|<\frac{m_k}k$
whenever $j\ge j(k)$. This means that 
\[
|I|=\left| \frac{\uej(x_k+\rho_k x)}{m_k}-\frac{u(x_k+\rho_k x)}{m_k}\right|<\frac1k.
\]
Consequently, we have that 
\[
\left| \frac{\uej(x_k+\rho_k x)}{m_k}-U(x)\right|\le |I|+|II|\le \delta +\frac1k, \quad x\in B_R.
\]
Observe that $j(k)$ can be  chosen so large that  
$\e_{j(k)}/m_k<\frac1k.$ Hence recalling \eqref{eq-lap-scale} and applying 
Proposition \ref{prop-compactness} we get parts (i) and (ii).
\smallskip

As for (iii) we use the estimate above to get $\|\na \hat u^{\e_{j}}-\na U\|_{L^p(B_R)}<\delta$ whenever $j>j(n)$. 
We have 
\[
\begin{split}
\|\na u_k-\na U\|_{L^p(B_R)}
\le 
\|\na u_k-\na \hat u^{\e_j}_k\|_{L^p(B_R)}+
\|\na \hat u^{\e_{j}}_k-\na U\|_{L^p(B_R)}\\
\le
\|\na u_k-\na \hat u^{\e_j}_k\|_{L^p(B_R)}+\delta.
\end{split}
\]
So it remains to estimate $\|\na u_k-\na \hat u^{\e_j}_k\|_{L^p(B_R)}$ for $j>j(k)$.
Let us estimate 
\[
\begin{split}
\|\na u_k-\na \hat u^{\e_j}_k\|_{L^p(B_R)}^p
=
\left[\frac{\rho_k}{m_k}\right]^p\int_{B_R}|\na u(x_k+\rho_kx)-\na \uej(x_k+\rho_k x)|^pdx\\
=
\left[\frac{\rho_k}{m_k}\right]^p\frac1{\rho_k^n}\int_{B_{\rho_kR}(x_k)}|\na u(x)-\na \uej(x)|^pdx.
\end{split}
\]
We know that $B_{\rho_kR}(x_k)\subset B_r(x_0)$ for large $k.$
Thus there is $k_0$ large such that 
\[
\int_{B_{r}(x_0)}|\na u(x)-\na \uej(x)|^pdx\le \delta \rho_k^n. 
\]
Therefore
\[
\|\na u_k-\na U\|_{L^p(B_R)}\le \delta +\left[\frac{\rho_k}{m_k}\right]^p\delta\le 2\delta,
\]
and (iii) follows.
\end{proof}

Finally, recall  that the result of previous proposition extends to the second blow-up.
\begin{prop}\label{prop-2nd-blow}
Let $u^{\e_j}$ be a solution to \eqref{pde-0} in a domain
$\mathcal D_j\subset \mathcal D_{j+1}$ and $\cup_j \mathcal D_j=\R^N$
such that $u^{\e_j}\to U$ uniformly on compact sets of $\R^N$  and
$\e_j\to 0$. Let us assume that for some choice of positive numbers $d_n$
and points $x_n\in \fb U$, the sequence
\[U_{d_n}(x)=\frac1{d_n}U(x_n+d_nx)\]
converges uniformly on compact sets of $\R^N$ to a function $U_0$.
Let
\[(u^{\e_j})_{d_n}=\frac1{d_n}u^{\e_j}(x_n+d_nx).\]
Then there exists $j(n)\to \infty$ such that for every $j_n\ge j(n),$
there holds $\e_{j_n}/d_n\to 0$ and
\begin{itemize}
\item $(u^{\e_{j_n}})_{d_n}\to U_0$ uniformly on compact subsets of
$\R^N$,
\item $\na (u^{\e_j})_{d_n}\to \na U_0$ in $L^2_{loc}(\R^N).$\end{itemize}
\end{prop}
\begin{proof}
See Lemma 3.3 \cite{CLW-uniform}.
\end{proof}

\section{Weak energy identity for the solutions of \eqref{pde-0} and the first domain variation}
\label{sec:weak-en}
This section contains the crucial tool for the proof of our main regularity theorem, the weak 
energy identity which we state below. 
In what follows we set $\mathcal B(t)=\int_0^t\beta(\tau)d\tau$. 
\begin{lem}
Let $\ue$ be a family of solutions to \eqref{pde-0}. For every $\phi=(\phi^1, \dots, \phi^n)\in C^1_0(B_1, \R^n)$ we have the identity.
\begin{equation}\label{putiin}
\int\left({|\na u^\e|^p}+p\mathcal B(\ue/\e)\right)\div \phi=p\sum_{m,l=1}^n\int |\na \ue|^{p-2}\p_l\ue \p_m \ue \p_l\phi^m.
\end{equation}
\end{lem}
\begin{proof}
Multiply $\Delta_p\ue=\beta_\e(\ue) $ by $\p_i\ue\phi$ and integrate
to get that 
\[
\int\mathcal B(\ue/\e) \phi^i_i=-\int\beta_\e(\ue)\ue_i\phi^i.
\]
On the other hand 
\begin{eqnarray*}
\int (|\na \ue|^{p-2}\ue_j)_j\ue_i\phi^i
&=&
- \int |\na \ue|^{p-2}\ue_j(\ue_{ij}\phi^i+u_i\phi^i_j)\\
&=&
\frac1p\int|\na \ue|^p\div \phi- \int |\na \ue|^{p-2}\ue_j\ue_i\phi^i_j. 
\end{eqnarray*}

\end{proof}

Next we prove that \eqref{putiin} is preserved in the limit as $\e\to 0$. 

\begin{lem}\label{putiin-2}
There is a bounded nonnegative function $0\le \mathcal B^\ast(x)\le M$ such that for every 
vector field $X\in C_0^1(B_1, \R^n)$ we have 
\begin{equation}\label{eq:energy}
\int \left[|\na u|^p+p\mathcal B^\ast(x))\right]\div X=p\int|\na u|^{p-2}\na u \na X\na u. 
\end{equation}

\end{lem}
\begin{proof}
We have $\mathcal B(\ue/\e)\to \mathcal B^\ast(x)$ $\ast$-weakly in $L^\infty_{loc}$.

By strong convergence of gradients, Proposition \ref{prop-comp} (iii),
\[
\int \left[|\na \ue|^p+p\mathcal B(\ue/\e))\right]\div X\to \int \left[|\na u|^p+p\mathcal B^\ast(x))\right]\div X.
\]
Let us show that the functions $|\na \uej|^{p-2}\na \uej \na X\na \uej$  are equiintegrable.
Given $\sigma>0$,  then there is $j_0$ and $\delta >0$ such that 
\[
\int_E|\na \uej|^{p-2}\left|\na \uej \na X\na \uej\right|<\sigma,
\]
whenever $|E|<\delta$ and $j>j_0$.
Indeed, we have
\[
\int_E|\na \uej|^{p-2}\left|\na \uej \na X\na \uej\right|\le \|\na X\|_\infty\left[\int_E|\na \ue|^p -\int_E|\na u|^p\right]
+
 \|\na X\|_\infty\int_E|\na u|^p.
\]
Choose $j_0$ so large that $\left|\int_E|\na \ue|^p -\int_E|\na u|^p\right|<\frac\sigma{2\|\na X\|_\infty}$ (which is possible thanks to Proposition \ref{prop-comp} (iii))
and then by the absolute continuity of the integral of $|\na u|^p$ we can choose 
$\delta$ so small that $\int_E|\na u|^p<\frac\sigma2$. Hence the desired result follows. 
\end{proof}

\subsection*{Domain variation formula for minimizers}
Let $\lambda>0$ be a constant. We show that the local minimizers of 
\begin{eqnarray}\nonumber
 J_p(u)=\int_{\Omega}|\nabla u|^p+\lambda^p\I{u}, 
\end{eqnarray}
satisfy 
\begin{equation}\label{putiin-3}
p\int_\Omega\left\{|\nabla u(y)|^{p-2}\cdot \na u(y)\na \phi(y)- \big[|\nabla u(y)|^p+\lambda(u)\big]\div\phi\right\}dy=0, 
\end{equation}
where, for the sake of simplicity, we set $\lambda(u)=\lambda^p \I{u}$.
The identity \eqref{putiin-2} is weaker than 
\eqref{putiin-3} since we do not know the explicit form of $\mathcal B^*$.

\begin{lem}
 Let $u$ be a local minimizer of $J_p(\cdot)$, then \eqref{putiin-3} holds. 

\end{lem}
\begin{proof}
Let $f(\xi)=|\xi|^p$.
For $\phi\in C^{0,1}_0(\Omega, \R^n)$ we put  $u_t(x)=u(x+t\phi(x))$,
with small $t\in\R$. Then
$\phi_t(x)=x+t\phi(x)$ maps $\Omega$ into itself. After change of variables $y=x+t\phi(x)$ we infer

\begin{eqnarray}\label{J-t}
\lefteqn{\int_\Omega \bigg[f(\nabla u_t(x))+\lambda(u_t(x))\bigg]dx=}\\\nonumber &
&=\int_\Omega\bigg[f(\nabla u_t(\phi^{-1}_t(y)))+
\lambda(u(y))\bigg]\bigg[1-t\div(\phi(\phi^{-1}_t(y)))+o(t)]\bigg]dy.\nonumber
\end{eqnarray}
Here we used the inverse mapping theorem for $\phi_t:x\to y$, in particular a well-known identity
\begin{eqnarray}
\left|\frac{D(x_1,\dots,x_n)}{D(y_1,\dots,y_n)}\right|=
\left|{\frac{D(y_1,\dots,y_n)}{D(x_1,\dots,x_n)}}\right|^{-1}=\frac{1}{1+t\div\phi+o(t)}.\nonumber
\end{eqnarray}
One can easily verify that
\begin{eqnarray}\nonumber
 \na u_{t}(x)=\na u(\phi_t(x))\left\{\mathbb I+ t\na \phi(x)\right\}
\end{eqnarray}
 with $\mathbb I=\{\delta_{ij}\}$ being the identity matrix. Hence

\begin{eqnarray}\nonumber
 \nabla u_t(\phi^{-1}_t(y))=\na u(y)\left\{\mathbb I+ t\na \phi(\phi^{-1}_t(y))\right\}.
\end{eqnarray}
and, moreover, 
\[
f(\na_x u_t(x))=f(\na_x u(y))+t\na _\xi f(\na_x u(y))\na_x u(y)\na \phi(\phi^{-1}_t(y))+o(t).
\]
This in conjunction with (\ref{J-t}) yields 
\begin{eqnarray}\nonumber
\int_\Omega
\Bigg\{\left[{\na f_\xi}
\big(\nabla u\left\{\mathbb I+ t\na \phi(\phi_t^{-1}(y))\right\} \big)
\cdot 
\na u(y)\na \phi(\phi^{-1}_t(y))\right]
\bigg[1-t(\div_x\phi)(\phi^{-1}_t(y))+o(t)]\bigg]\\\nonumber
-\bigg[f(\nabla u_t(\phi^{-1}_t(y)))+\lambda(u(y))\bigg]\bigg[(\div_x \phi)(\phi^{-1}_t(y))+o(1)]\bigg]
\Bigg\}dy
\longrightarrow\\\nonumber
\longrightarrow \int_\Omega\left\{{\na_\xi f}(\nabla
u)\cdot \na u(y)\na \phi(y)- \big[f(\nabla
u(y))+\lambda(u)\big]\div\phi\right\}dy=0.
\end{eqnarray}
This completes the proof of \eqref{putiin-3}.
\end{proof}

\medskip

It is convenient to introduce the {\it variational solutions} of the free boundary
problem %
\begin{eqnarray}\label{E-L-p}
&\Delta_pu=0\ \mbox{in}\  \Omega\cap\{u>0\}\cup\{u<0\},%
\\\nonumber
 & |\na u^+|^p-|\na u^-|^p=\frac{\lambda^p}{p-1}\  \mbox{on}\  \Omega\cap \partial \{u>0\}.
\end{eqnarray}

\smallskip

\begin{defn}
Let $f(\xi)=|\xi|^p$, then a function
$u\in W^{1, p}(\Omega)$ is said to be a variational
solution of (\ref{E-L-p}) in some domain $\Omega$ if $\Delta_pu=0$ in 
$ \Omega\cap\{u>0\}\cup\{u<0\}$ in weak sense,  and
for any $\phi\in C^{0,1}_0(\Omega, \R^n)$
\begin{eqnarray}\label{stat-point}
 \int_\Omega\left\{{\na_\xi f}(\nabla
u)\cdot \na u(y)\na \phi(y)- \big[f(\nabla
u(y))+\lambda(u)\big]\div\phi\right\}dy=0.
\end{eqnarray}
\end{defn}

\begin{rem}
 By inspecting the proof of Theorem \ref{thm-lip} one can see 
 that the local Lipschitz estimate is valid for the variational solutions in $B_1$ provided that 
 $p>n$, {\bblue the latter is a technical condition to assure uniform continuity of 
 $u^\e$, see the discussion in Section \ref{sec:tools} and Remark \ref{xrusch}}. Also observe  that every stationary point of
$J_p$ is a variational solution as the above computation shows.
\end{rem}

\section{Viscosity solutions}\label{sec:visc-space}

\begin{thm}\label{thm-entropy}
Let $\uej\in W^{1,p}_{loc}(\R^n), \na \uej\in BMO_{loc}(\R^n)$ solve 
$\Delta_p \uej=\beta_j(\uej)$ with $\beta_j(t)=\sigma_j\frac1{\e_j}\beta(\frac t{\e_j})$
and $\sigma_j\downarrow  0$. 
Let $\uej\to u$ in $W^{1,p}_{loc}(\R^n), \na u\in BMO_{loc}(\R^n)$ such that $\Delta_p u=0$ in $\po u\cup\{u<0\}$. 
Then 
$\Delta_p u=0$ in $\R^n.$
\end{thm}

\begin{proof}
Let $u(x_0)>0$ and let $B_r(x_0)\subset \{u>0\}$ such that $y_0\in \p B_r(x_0)$
for some $y_0\in \fb u$.
Then $u^-$ is bounded by the barrier $b$,  $\Delta_p b=0$ in $B_{2r}(x_0)\setminus 
B_r(x_0)$, $b=0$ on $\p B_r(x_0)$ and $b=\max_{B_{2r}(x_0)} u^-$ on $\p B_{2r}(x_0)$.  For any $s\in (0, r)$ we have 
\[
\fint_{B_s(y_0)} u=\int_0^s\left[\fint_{B_\tau(y_0)}\na u(x)\cdot \frac {x-y_0}{|x-y_0|}dx\right]d\tau=
\int_0^s\fint_{B_\tau(y_0)}\pscal{\na u(x)-\fint_{B_\tau(x_0)} \na u}{\frac {(x-y_0)}{|x-y_0|}}dxd\tau, 
\]
{\bblue where the last equality follows from the observation that 
$\fint_{B_\tau(y_0)}\pscal{e}{\frac {(x-y_0)}{|x-y_0|}}dx =0
$ for any fixed vector $e$. }
Consequently, 
\begin{eqnarray}
\fint_{B_s(y_0)} u^+
&=&
\fint_{B_s(y_0)} u^- -\int_0^s\fint_{B_\tau(y_0)}\pscal{\na u(x)-\fint_{B_\tau(x_0)} \na u}{ \frac{ (x-y_0)}{|x-y_0|}}dxd\tau\\\nonumber
&\le& 
\fint_{B_s(x_0)} b+s\|\na u\|_{BMO}
\\\nonumber
&\le&
s(\|\na b\|_\infty+ \|\na u\|_{BMO}).
\end{eqnarray}
Thus if $u_0$ is a blow-up $u$ at $y_0$ then from Lamma 4.3 \cite{DK} (see also 
Lemma \ref{lemma:linear})   
either $u_0(x)=\alpha|x_1|,$ for some  $\alpha\ge 0$ in a suitable coordinate system or $u_0$ is linear. Thus to show that $u$ is a viscosity solution at 
$y_0$ it is enough to conclude that $\alpha=0$.

We can use the construction in \cite{CLW-uniform} of $x_1$ symmetric solution and assume that 
$\uej$ is $x_1$ symmetric.  Hence, thanks to Propositions \ref{prop-1st-blow}
and \ref{prop-2nd-blow},  from now on we assume that 
$\uej$ is a family of solutions such that $\uej$ converges to $\alpha|x_1|$ is some suitable coordinate system. We claim that $\alpha=0$. 
To see this we observe that 
\begin{equation}
\p_k(T_{1k})=\p_j(p|\na \uej |^{p-2}\p_k\uej \p_1\uej-\delta_{1, k}|\na \uej |^p)=\p_1\mathcal B. 
\end{equation}
Let us denote
\[
\mathcal R_t=\{(x_1, x') : 0<x_1, |x'|<t \},
\]
and
\begin{equation}
E_j\eqdef \int_{\mathcal R_t} p\p_1(|\na \uej|^{p-2}(\p_1\uej)^2-|\na \uej|^p)- \p_1\mathcal B_j(\uej)
=-p\int_{\mathcal R_t}\p_j\left[ \p_1\uej\sum_{k=2}^n |\na \uej|^{p-2}\p_k\uej\right]\eqdef -G_j.
\end{equation}
From the divergence theorem and $\p_1 \uej=0$ on $x_1=0$, we get 
\begin{eqnarray*}
E_j
&=&
\int_{\{x_1=0\}\cap \mathcal R_t}(-\mathcal B_j-|\na _{x'}\uej|^{p})(-1)+
\int_{\{x_1=1\}\cap \mathcal R_t}(p|\na \uej|^{p-2}(\p_1\uej)^2-\mathcal B_j-|\na \uej|^p)\\
&\ge&
\int_{\{x_1=1\}\cap \mathcal R_t}(p|\na \uej|^{p-2}(\p_1\uej)^2-\mathcal B_j-|\na \uej|^p).
\end{eqnarray*}
On the other hand 
\begin{eqnarray*}
-G_j=-p\int_{\{|x'|=t\}}\left[ \p_1\uej\sum_{k=2}^n |\na \uej|^{p-2}\p_k\uej\right]\nu^k\le p\int_{\{|x'|=t\}}\left[ |\p_1\uej|\sum_{k=2}^n |\na \uej|^{p-2}|\p_k\uej|\right].
\end{eqnarray*}
Integrate the inequality 
\[
\int_{\{x_1=1\}\cap \overline{\mathcal R_t}}(p|\na \uej|^{p-2}(\p_1\uej)^2-\mathcal B_j-|\na \uej|^p)\le p\int_{\{|x'|=t\}}\left[ |\p_1\uej|\sum_{k=2}^n |\na \uej|^{p-2}|\p_k\uej|\right]
\]
over $t\in [1-\delta, 1]$ and using the co-area  formula we conclude that 
\begin{eqnarray}\label{eq:impichment}
\int_{[1-\delta, 1]}dt \int_{\{x_1=1\}\cap \overline{\mathcal R_t}}(p|\na \uej|^{p-2}(\p_1\uej)^2-\mathcal B_j-|\na \uej|^p)\\\nonumber
\le 
p\int_{(B_1'\setminus B_{1-\delta}')\times [0, 1]}\left[ |\p_1\uej|\sum_{k=2}^n |\na \uej|^{p-2}|\p_k\uej|\right].
\end{eqnarray}
Note that 
\[
p|\na \uej|^{p-2}(\p_1\uej)^2-\mathcal B_j-|\na \uej|^p\to (p-1)\alpha^p, 
\]
whereas 
\[
|\p_1\uej|\sum_{k=2}^n |\na \uej|^{p-2}|\p_k\uej|\to 0
\]
pointwise in $x_1>0$. Hence from \eqref{eq:impichment} it follows that 
\[
(p-1)\alpha^p\le 0, 
\]
and the claim follows. Thus $u$ is a viscosity solution of $\Delta_p u=0$ in $\R^n$.
\end{proof}

\section{Dyadic scaling: Flatness vs linear growth}
\label{sec:dyadic}
\subsection{Slab-flatness}
This and next sections contain the main ingredients for the proof of the 
local Lipschitz estimate. 
The free boundary points can be characterized by the modulus of 
continuity $\delta $ of the slab flatness at $x_0\in \fb u$ and the 
Labasgue density $\Theta $ of $\{u<0\}$. The good points for the linear growth 
are those where $\delta$ and $\Theta$ are not very small. In this
section we deal with the points where 
$\delta$ is not small, and we show that 
at such points $u$ grows linearly.  

In order to formulate the main result of this section we introduce the 
notion of slab flatness for $\fb u$.

Let $x_0\in \fb u$ and
\begin{equation}
\mathcal S(h; x_0, \nu):=\{x\in \R^n : -h<(x-x_0)\cdot\nu<h\}
\end{equation}
be the slab of height $2h$ in unit direction $\nu$. 
Let $h_{\min}(x_0, r, \nu)$ 
be the the distance of two parallel planes  with unit direction $\nu$ 
containing the free boundary  $\fb u$ in $B_r(x_0)$, i.e.
\begin{equation}
h_{\min}(x_0, r, \nu):=\inf\{h : \partial\{u>0\}\cap B_r(x_0)\subset S(h; x_0, \nu)\cap B_r(x_0)\}.
\end{equation}
Finally,  let 
\begin{equation}\label{min-h}
h(x_0, r):=\inf_{\nu\in \mathbb S^N}h_{\min}(x_0, r, \nu).
\end{equation}
Note that  $h(x_0, r)$ is {\bf non-decreasing} in $r$. 
We call $h(x_0, r)/2$ {\bf the slab flatness constant} at scale $r>0$.

\smallskip

\medskip
\subsection{Optimal growth}
\begin{prop}\label{prop:lin flat}
Let~$\uej$ be a  family of solutions  to \eqref{pde-0}, such that $\uej\to u$
locally uniformly in $B_{1}$.  
Let $x_0\in\Gamma\cap B_{1/2}$. For any~$k\in\N$, set
$$
S(k,u):=\sup_{ { { B_{2^{-k}} (x_0)}}}|u|.
$$
 
If $h_0>0$ is fixed and $h\left(x_0,\frac1{2^k}\right)\ge \frac{h_0}{2^{k+1}}$ for some~$k$, then 
\begin{equation}\label{discrete-linear}
S(k+1,u)\leq\max\left\{\frac{L2^{-k}}{2},\frac{S(k,u)}{2},\dots, 
\frac{S(k-m,u)}{2^{m+1}},\dots, \frac{S(0,u)}{2^{k+1}}\right\},
\end{equation}
for some positive constant $L$, that is independent of~$x_0$ and $k$. 
\end{prop}

\begin{proof} 
Suppose that the assertion of proposition is false.
Then  
there exist integers~$k_j, j=1,2,\ldots$, limits  ~$u_j,\sup_j|u_j|\le 1$ (i.e. $u_j=\lim_{k\to \infty} u^{\e_k(j)}$
where $\{u^{\e_k(j)}\}_{k=1}^\infty$ are solutions to \eqref{pde-0} for each $j$ fixed)
and points~$x_j\in\Gamma_j\cap B_1$ such that 
\begin{equation}\label{level}
h\left(x_j,\frac1{2^{k_j}}\right)\ge \frac{h_0}{2^{k_j+1}}
\end{equation}
and
\begin{equation}\label{neg}
S(k_j+1,u_j)>\max\left\{\frac{j2^{-k_j}}{2},\frac{S(k_j,u_j)}{2},\dots, 
\frac{S(k_j-m,u_j)}{2^{m+1}},\dots,\frac{S(0,u_j)}{2^{k_j+1}}\right\}.
\end{equation}

Therefore, from~\eqref{neg} 
we have that~$1\ge j2^{-k_j}/2$,
which implies that~$2^{k_j}\ge j/2$. 
Hence, $k_j$ tends to~$+\infty$ when~$j\to+\infty$. 

We set
\begin{equation}\label{sigma}
\sigma_j:=\frac{2^{-k_j}}{S(k_j+1,u_j)}. 
\end{equation}
It follows from \eqref{neg}  that 
\begin{equation}\label{cc18}
\sigma_j <\frac{2}{j}\to 0 \ {\mbox{ as }}j\to+\infty.
\end{equation}

The basic idea of the proof is to show that the scaled functions 
\begin{equation}\label{scaled v}
v_j(x):=\frac{u_j(x_j+2^{-k_j}x)}{S(k_j+1,u_j)}
\end{equation}
converge to a linear function in $\R^n$, which will be in contradiction with \eqref{level}. 
The proof falls naturally into two parts: first establish some uniform estimates for 
the sequence $\{v_j\}$, and then prove that the limit is a $p$-harmonic function in $\R^n$.

By construction, 
\begin{equation}\label{cc1}
\sup_{B_{1/2}}|v_j|=1.
\end{equation}
Furthermore, from~\eqref{neg} we have that 
$$ 1>\max\left\{\frac{j2^{-k_j}}{2S(k_j+1,u_j)},\frac{1}{2}\sup_{B_1}|v_j|,\dots, 
\frac{1}{2^{m+1}}\sup_{B_{2^m}}|v_j|,\dots,\frac{1}{2^{k_j+1}}\sup_{B_{2^{k_j+1}}}|v_j|\right\},$$
which in turn implies that 
\begin{equation}\label{cc2}
\sup_{B_{2^m}}|v_j|\le 2^{m+1}, \quad \text{for any}\ m<2^{k_j}.
\end{equation}
Finally, since~$u_j(x_j)=0$, we have that
\begin{equation}\label{cc3} v_j(0)=0.\end{equation}

Next, from~\eqref{scaled v} we get 
$$ \nabla v_j(x)=\frac{2^{-k_j}}{S(k_j+1,u_j)}\nabla u_j(x_j+2^{-k_j}x)=\sigma_j\na u_j(x_j+2^{-k_j}x).$$
This gives 

\begin{eqnarray*}
T_{lm}(\na u_j)(x_j+2^{-k_j}x)=T_{lm}(\frac1{\sigma_j}\na v_j(x))=\frac1{\sigma_j^p}T_{lm}(\na v_0(x)).
\end{eqnarray*}
Consequently, letting $\mathcal B_j^\ast(x):=\mathcal B^\ast(x_j+2^{-k_j}x)$ 
and substituting $\na v_j$ into \eqref{eq:energy} we get the differential relation 

\begin{equation}\label{eq:T-eq}
\p_l(T_{lm}(\na v_j))=\p_m(\sigma_j^p p \mathcal B_j^\ast(x)).
\end{equation}
Note that $\sigma_j^p\mathcal B_j^*(x)\to 0$
since $\mathcal B^*$ is bounded.

Hence, from Propositions~\ref{prop-comp},  \ref{prop-1st-blow} (with $m_j=S(k_j+1, u_j)$) 
and \ref{prop-seq-comp} we obtain that
for any $0<R<2^{k_j}$ there exists a 
constant~$C=C(R,p)>0$ independent of~$j$ such that 
$$
\max\{\|v_j\|_{C^{\alpha}(B_R)},\|\nabla v_j\|_{L^p(B_R)}\}\le C,
$$
with  {\bblue{$\alpha>0$ because $\|\nabla v_j\|_{BMO(B_R)}\le C\|v_j\|_{L^\infty(B_{2R})}\le C2^{m+2}$ in view of \eqref{cc2} and \eqref{BMO} (one can take $\alpha=1-\frac np$, if $p>n$).}}
Therefore, by a standard compactness argument, we have that, up to a subsequence,
\begin{equation}\begin{split}\label{choice q}
&{\mbox{$v_j$ 
converges to some function~$v$ as~$j\to+\infty$ in~$C^{\alpha}(B_R)$}}, \\
&{\mbox{and
strongly in~$W^{1,p}(B_R)$, for any fixed~$R$.}}
\end{split}\end{equation} 
 {\bblue Note that the estimate  $\|\nabla v_j\|_{BMO(B_R)}\le C2^{m+2}$, which holds for any 
 fixed $m$ satisfying $R<2^{m+2}$, implies that 
 $\|\na v_j\|_{L^{p'}(B_{R'})}\le C(p',R')$ uniformly for any $R'<R<2^{m+2}$ and $p'<\infty$. }
{\bblue Consequently, by $C^\alpha$ uniform continuity} the properties \eqref{cc1}, \eqref{cc2} and~\eqref{cc3} translate to 
$v$ so that  
\begin{equation}\label{xrusch1101}
\sup_{B_{1/2}}|v|=1,\quad \sup_{B_{2^m}}|v|\le 2^{m+1} \quad 
{\mbox{and }} \quad v(0)=0.
\end{equation}

Thanks to Theorem \ref{thm-entropy}, \eqref{eq:T-eq} and the strong convergence 
$\na v_j\to \na v$ in {\bblue $L^{p'}_{loc}(B_R')$ for any $p'<\infty$}, we conclude that  $\Delta_p v=0$ in $\R^n$.

{\bblue Given $x\in R^n$, take $m\in \mathbb N$ such that $2^m\le |x|\le 2^{m+1}$, 
then utilizing the second inequality in \eqref{xrusch1101} we get $|v(x)|\le \sup_{B_{2^{m+1}}}|v|\le C2^{m+2}\le 4C|x|$. }
Hence, from Liouville's Theorem {\bblue for the $p-$harmonic functions in $\R^n$ with linear growth}, we deduce that~$v$ must be a linear function in $\R^n$.
After rotating the coordinate system we can take 
\begin{equation}\label{sper23}
{\mbox{$v(x)=Cx_1$ for some positive constant
$C$.}}\end{equation} 
{\bblue Note that $\sup_{B_{1/2}}|v|$ in \eqref{xrusch1101} implies that $C\not=0$. }

On the other hand, \eqref{level} implies that the following inequality holds true 
for the function~$v_j$: 
$$ h(0,1)\ge \frac{h_0}{2}.$$
By the uniform convergence in \eqref{choice q}, 
we have that for any $\epsilon >0$ there is $j_0$ such that $|Cx_1-v_j(x)|<\epsilon$
whenever $j>j_0$.  Since $\fb {v_j}$ is $h_0/2$ thick in $B_1$ it follows that 
there is $y_j\in \fb {v_j}\cap B_1$ such that $y_j=e_1h_0/4+t_je'$, for some $t_j\in \R$, where 
$e_1$ is the unit direction of $x_1$ axis and $e'\perp e_1$. Then we have that 
$\left|C\frac{h_0}4-0\right|=|v(y_j)-v_j(y_j)|<\epsilon$, which is a contradiction if $\epsilon$ is small.
This finishes the proof of \eqref{discrete-linear}.
\end{proof}

\section{Density of $\{u<0\}$ vs linear growth}
\label{sec:dens}
Define 
\[
\Theta(u, x_0, r)=\frac{\vol(\{ u<0\}\cap B_r(x_0))}{\vol( B_r)}.
\]
In this section we prove 
\begin{lem}\label{lem-dense}
There is $\delta>0$ such that if $\Theta(u, x_0, r)<\delta$ for some $B_r(x_0)\subset B_{1/2}, x_0\in \fb {u_0}$
then 
\[
\sup_{B_{\frac r2}(x_0)}|u|\le \frac{4 r}{\delta}.
\]
\end{lem}
\begin{proof}
Suppose that 
\begin{equation}\label{dens-neg}
\Theta(u, x_0, 2^{-k})<\delta 
\end{equation} 
and we claim that 
\begin{equation}\label{7-dndooo}
S(k+1)\le \max  \left\{ \frac{1}{\delta 2^{k+1}}, \frac12 S(k) \right\},
\end{equation} 
where $S(k):=\sup_{B_{2^{-k}}(x_0)}|u|$, for $k\in\N$. 
The proof by contradiction.
Suppose that \eqref{7-dndooo} fails. Then 
there is a sequence of integers $k_j$ and $u_j$ (i.e. $u_j=\lim_{k\to \infty} u^{\e_k(j)}$
where $\{u^{\e_k(j)}\}_{k=1}^\infty$ are solutions to \eqref{pde-0} for each $j$ fixed), with $j=1,2,\ldots$, such that 
\begin{equation}\label{7-not-dndooo}
\Theta(u_j, x_j, 2^{-k_j})\le \frac1j, \quad S(k_j+1)> \max  \left\{\frac{ j}{2^{k_j+1}}, \frac12 S(k_j) \right\}.
\end{equation} 
Since $|u_j|\le 1$, then \eqref{7-not-dndooo} implies that 
$k_j\to \infty$ as $j\to+\infty$. 
Also, notice that \eqref{7-not-dndooo} implies that 
\begin{equation}\label{7-3.30bis}
\frac{2^{-k_j}}{S(k_j+1)}\le\frac{2}{j}\to 0 \quad {\mbox{ as }} j\to+\infty.
\end{equation}

Now, we introduce the scaled functions $v_j(x):=\frac{u_j(x_0+2^{-k_j}x)}{S(k_j+1)}$, for  $x\in B_1$.
Then, from \eqref{dens-neg} and \eqref{7-3.30bis}, it follows that 
\begin{equation}\label{7-dndooo-3}
\Theta (v_j, 0, 1)\le \frac 1j, \quad v_{j}(0)=0. 
\end{equation}

Furthermore,  \eqref{7-not-dndooo} implies that 
\begin{equation}\label{7-dndooo-1}
\sup_{B_1}|v_j|\leq 2, \quad {\mbox{ and }} \quad \sup_{B_{\frac12}}|v_j|=1.
\end{equation}
We know from Proposition \ref{prop-seq-comp} with 
$m_j=S(k_j+1)$ that $\|v_j\|_{W^{1,p}(B_{\frac34})}$ are uniformly bounded. 
So we can extract 
a converging subsequence such that $v_j\to v_0$
uniformly in $\overline{B_{\frac34}}$ 
and $\na v_j\to \na v_0$ weakly
in $L^p(B_{\frac34})$. 
As in the proof of Proposition \ref{prop:lin flat} we get that $\Delta_p v_0 \na v_0=0 $, 
and consequently $v_0$ is a $p$-harmonic function in $\R^n$.

Moreover,  \eqref{7-dndooo-3}, \eqref{7-dndooo-1} and Theorem \ref{thm-entropy} yield
\begin{eqnarray*}
\Delta_p v_0(x)=0, \quad v_0(x)\ge 0\  \text{if}\ x\in B_{\frac34}, \quad  v_0(0)=0, 
\quad {\mbox{ and }} \sup_{B_{\frac12}}v_0=1
\end{eqnarray*}
which is in contradiction with the strong minimum principle. 
This shows \eqref{7-dndooo} and the proof follows.

\end{proof}

\section{Viscosity solutions}
\label{sec:viscosity}

If $\Theta(u, x_0, r)$ is either small or $\fb u$ is non-flat then $u$ has linear growth near $x_0\in \fb u$.
Thus the remaining case to be analyzed is the following : 
\[
\Theta(u, x_0, r)\  \text{is large {\bf and}} \ \fb u\ \text{ is flat near}\  x_0.
\]

To tackle this remaining case we want to use the stratification argument from \cite{DK} for the 
viscosity solutions 
in order to obtain the Lipschitz continuity of $u$. 
This will be done by combining the above results. 
To define the notion of viscosity solution we let 
$\Omega^+(u)=\{u>0\}$ and~$\Omega^-(u)=\{u<0\}$.
If the free boundary is $C^1$  smooth then 
\begin{equation}\label{FB-cond-G}
G(u^+_\nu,u^-_\nu):=(u_\nu^+)^p-(u_\nu^-)^p-\Lambda_0
\end{equation} 
is called the free boundary condition,
where~$u^+_\nu$ and~$u^-_\nu$ are the normal derivatives in the inward direction
to~$\partial \Omega^+(u)$ and~$\partial \Omega^-(u)$, respectively. 
Here $\Lambda_0=\frac{\Lambda}{p-1}=\frac{ pM}{p-1}$ is the Bernoulli constant with $M=\int_0^1\beta$. 

To justify the form of the free boundary condition we first show that 
for smooth free boundaries \eqref{FB-cond-G} is true. 

\begin{lem}\label{lem:8}
Let~$\uej$ be a  family of solutions  to \eqref{pde-0}, with $\e=\e_j$, such that $\uej\to u$
locally uniformly in $B_{1}$.  Suppose that $\fb u$ is $C^{1, \gamma}, \gamma\in (0, 1)$ 
regular hypersurface and $\Theta(u, x_0, r)>\delta$ for some $r>0$. Then 
\eqref{FB-cond-G} holds. 
\end{lem}
\begin{proof}
Let $x_0\in \fb u$, then from the boundary estimates for the 
$p$-harmonic functions we know that $u^\pm$ are $C^{1, \gamma}$ up to the free boundary.
Let $\rho_k\to 0$ and consider $\frac1{\rho_k}u(x_0+\rho_kx)\to \alpha x_1^+-\bar\alpha x_1^-$
(after rotation the coordinate system), where $\alpha, \bar \alpha$ are nonnegative constants. By Proposition \ref{prop-1st-blow}
there is a sequence $\e_j\to 0$ so that $u^{\e_j}(x)\to \alpha x_1^+-\bar\alpha x_1^-$ uniformly on 
compact subsets of $\R^n$
and $\e_j/\rho_k\to 0$. Moreover, $\na u^{\e_j}\to \alpha e_1\I u-\bar\alpha e_1^-\chi_{\{u\le 0\}}$ 
strongly in $L^p$ on compact subsets of $\R^n$.

Let us check that $\mathcal B((\uej/\rho_k)/(\e_j/\rho_k))=\int_0^{\uej/\e_j}\beta(t)dt \to M\I u$
$\ast$-weakly in $L^\infty_{loc}$.  Take $\tau>0$ small and fix $R>0$, then 
\begin{eqnarray*}
\int_{B_R} \mathcal B((\uej/\rho_k)/(\e_j/\rho_k))
&=&
\int_{B_R\cap\{|x_1|<\tau\}} \mathcal B((\uej/\rho_k)/(\e_j/\rho_k))\\
&&+
\int_{B_R\cap\{|x_1|\ge\tau\}} \mathcal B((\uej/\rho_k)/(\e_j/\rho_k))=I_1+I_2.
\end{eqnarray*}
By uniform convergence $\uej\to \alpha x_1-\bar\alpha x_1^-$ we have that 
there exists $j_0$ large so that $\uej/\e_j>1$ in $B_R\cap\{|x_1|\ge\tau\}$, thus 
$I_2=\int_{B_R\cap\{|x_1|\ge\tau\}}M\to M\int_{B_R\cap\{|x_1|\ge\tau\}}$.
As for the remaining term $I_1$, we observe that
\[
I_1\le \tau R^{n-1}.
\] 
Thus first sending $j\to \infty$ and then $\tau\to 0$ the desired result follows.  
Consequently we can apply \eqref{eq:energy} to 
$ \alpha x_1^+-\bar\alpha x_1^-$ to obtain 
\[
\int_{B_1^+}(\alpha^p+pM) \div X+\int_{B_1^-}(\bar \alpha^p+M) \div X
=
p\int_{B_1^+}\alpha^p X_1^1 +p\int_{B_1^-}(\bar \alpha^p+M) X_1^1,
\]
where we used the notation $B^\pm_1=B_1\cap \{\pm x_1>0\}$.
Since $\int_{B_1^\pm}\div X=\mp\int_{B_1\cap \{x_1=0\}}X^1$
and $\int_{B_1^\pm} X^1_1=\mp\int_{B_1\cap \{x_1=0\}}X^1$
we get 
\[
-(\alpha^p+pM)+\bar \alpha^p=-p\alpha^p+p\bar\alpha^p,
\]
or equivalently $(\alpha^p-\bar\alpha^p)(p-1)=pM$ which is \eqref{FB-cond-G}.
\end{proof}

\begin{defn}\label{def:visc}
Let~$\Omega$ be a bounded domain of~$\R^N$ and let~$u$ 
be a continuous function in~$\Omega$. We say that~$u$ is a viscosity solution
in~$\Omega$ if
\begin{itemize}
\item[i)] $\Delta_p u=0$ in~$\Omega^+(u)$ and~$\Omega^-(u)$,
\item[ii)] along the free boundary~$\Gamma$, $u$ satisfies the free boundary condition, in the sense that:
\begin{itemize}
\item[a)] if at~$x_0\in\Gamma$ there exists a ball~$B\subset\Omega^+(u)$
such that~$x_0\in \partial B$ and
\begin{equation}\label{visc1}
u^+(x)\ge\alpha\langle x-x_0,\nu\rangle^+ + o(|x-x_0|), \ {\mbox{ for }} x\in B,
\end{equation}
\begin{equation}\label{visc2}
u^-(x)\le\beta\langle x-x_0,\nu\rangle^- + o(|x-x_0|), \ {\mbox{ for }} x\in B^c,
\end{equation}
for some $\alpha>0$ and~$\beta\ge0$, with equality along every non-tangential domain,
then the free boundary condition is satisfied
$$ G(\alpha,\beta)=0, $$
\item[b)] if at~$x_0\in\Gamma$ there exists a ball~$B\subset\Omega^-(u)$
such that~$x_0\in \partial B$ and
$$ u^-(x)\ge\beta\langle x-x_0,\nu\rangle^- + o(|x-x_0|), \ {\mbox{ for }} x\in B, $$
$$ u^+(x)\le\alpha\langle x-x_0,\nu\rangle^+ + o(|x-x_0|), \ {\mbox{ for }} x\in\partial B, $$
for some $\alpha\ge0$ and~$\beta>0$, with equality along every non-tangential domain,
then
$$ G(\alpha,\beta)=0. $$
\end{itemize}
\end{itemize}
\end{defn}

The main result of this section is the following:
\begin{thm}\label{TH:viscosity}
Let~$\uej$ be a  family of solutions  to \eqref{pde-0}, such that $\uej\to u$
locally uniformly in $B_{1}$.  
Then~$u$ is a viscosity solution in~$\Omega$ in the sense of 
Definition~\ref{def:visc}.
\end{thm}

The proof of  Theorem~\ref{TH:viscosity}, will follow from  Lemmas \ref{lemma:linear} and \ref{lem:dndooo} below. 
For the proof of Lemma \ref{lemma:linear} see Appendix \cite{DK}.

\begin{lem}\label{lemma:linear}
Let~$0\le u\in W^{1, p}(\Omega)$ be a solution of $\Delta_p u=0$ in~$\Omega$
and~$x_0\in\partial\Omega$. Suppose that~$u$ continuously 
vanishes on~$\partial\Omega\cap B_1(x_0)$.
Then
\begin{itemize}
\item[a)] if there exists a ball~$B\subset\Omega$ touching~$\partial\Omega$ at~$x_0$,
then either~$u$ grows faster than any linear function at~$x_0$, or there exists
a constant~$\alpha>0$ such that
\begin{equation}\label{linear alpha}
u(x)\ge \alpha\langle x-x_0,\nu\rangle^+ +o(|x-x_0|) \quad {\mbox{ in }}B,
\end{equation}
where~$\nu$ is the unit normal to~$\partial B$ at~$x_0$, inward to~$\Omega$.
Moreover, equality holds in~\eqref{linear alpha} in any non-tangential domain.
\item[b)] if there exists a ball~$B\subset\Omega^c$ touching~$\partial\Omega$ at~$x_0$,
then there exists a constant~$\beta\ge0$ such that
\begin{equation}\label{linear beta}
u(x)\le \beta\langle x-x_0,\nu\rangle^+ +o(|x-x_0|) \quad {\mbox{ in }}B^c,
\end{equation}
with equality in any non-tangential domain.
\end{itemize}
\end{lem}

With this, we are able to prove Theorem~\ref{TH:viscosity} by utilizing the following anisotropic scaling 
argument. 

\begin{thm}\label{thm-three}
Let $B\subset \Om^+$ be a touching ball to $\Gamma$ from $\{u>0\}$ (resp. $\{u<0\}$) then 
in the asymptotic expansions \eqref{linear alpha} and \eqref{linear beta}
both $\alpha$ and  $\beta$ are finite and uniformly bounded.
\end{thm}
\begin{rem}
From Theorem \ref{thm-three} it follows that the limit $u$ is a viscosity solution in the sense of Definition \ref{def:visc}.
\end{rem}

We recapitulate the statement of Theorem \ref{thm-three}
and amplify it by proving a more quantitative result. 
It can be proven in much the same way as Lemma \ref{lem-dense}. We give only the main ideas of the proof.

\begin{lem}\label{lem:dndooo}
 Let~$\uej$ be a  family of solutions  to \eqref{pde-0}, such that $\uej\to u$
locally uniformly in $B_{1}$.  
Let $x_0\in \fb u$ and $r>0$ small such that $B_r(x_0)\subset \Omega$. 
Assume that $\sup_{B_r(x_0)}u^-\leq C_0 r$
(resp. $\sup_{B_r(x_0)}u^+\leq C_0 r$) $\forall r\in(0, r_0)$, 
for some constant $C_0$ depending on $x_0$ and $ r_0$ small.

Then there exists a constant $\sigma>0$ such that $\sup_{B_r(x_0)}u^+\leq (1+\sigma C_0) r$  (resp. 
$\sup_{B_r(x_0)}u^-\leq (1+\sigma C_0) r$).
\end{lem}
\begin{rem}
Lemma \ref{lem:dndooo} implies that $u^+$ and $u^-$ have coherent growth. This implies that 
if $u$ is as in Theorem \ref{thm-three} then the scaled functions $\frac{u(x_0+rx)}r$ converge to the 
half plane solution $\alpha x_1^+-\beta x_1^-$ in an appropriate coordinate system. 
\end{rem}

\begin{proof}
We will show only one of the claims, the other can be proved analogously.
Suppose that 
\begin{equation}\label{cx-0-hav}
\sup_{B_r(x_0)}u^-\leq C_0 r, 
\end{equation} 
and we claim that 
\begin{equation}\label{dndooo}
S(k+1)\le \max  \left\{ \frac{ 1+\sigma C_0}{2^{k+1}}, \frac12 S(k) \right\},
\end{equation} 
where $S(k):=\sup_{B_{2^{-k}}(x_0)}|u|$, for any $k\in\N$. 
To prove this, we argue by contradiction and we
suppose that \eqref{dndooo} fails. Then 
there is a sequence of integers $k_j$, with $j=1,2,\ldots$, such that 
\begin{equation}\label{not-dndooo}
S(k_j+1)> \max  \left\{\frac{ j}{2^{k_j+1}}, \frac12 S(k_j) \right\}.
\end{equation} 
From the bound $\|u\|_\infty\le 1$ and \eqref{not-dndooo} it follows that 
$k_j\to \infty$ as $j\to+\infty$. 
Also, notice that \eqref{not-dndooo} implies that 
\begin{equation}\label{3.30bis}
\sigma_j:=\frac{2^{-k_j}}{S(k_j+1)}\le\frac{2}{j}\to 0 \quad {\mbox{ as }} j\to+\infty.
\end{equation}

Now, we introduce the scaled functions $v_j(x):=\frac{u(x_0+2^{-k_j}x)}{S(k_j+1)}$, for  $x\in B_1$.
Then, from \eqref{cx-0-hav} and \eqref{3.30bis}, it follows that 
\begin{equation}\label{dndooo-3}
v_{j}(0)=0 \quad {\mbox{ and }} \quad 
v_j^-(x)=\frac{u^-(x_0+2^{-k_j}x)}{S(k_j+1)}\leq \frac{2^{-k_j} C_0}{S(k_j+1)}<\frac{2C_0}{j}\to 0 
\; {\mbox{ as }} j\to+\infty.
\end{equation}

Furthermore, it is not difficult to see that \eqref{not-dndooo} implies that 
\begin{equation}\label{dndooo-1}
\sup_{B_1}|v_j|\leq 2, \quad {\mbox{ and }} \quad \sup_{B_{\frac12}}|v_j|=1.
\end{equation}

\[
\int_{B_1}\left [\na v_j|^p+\sigma_j^p p \mathcal B^*(S(k_j+1)2^{k_j}v_j)\right ]\div \psi=p\int_{B_1}|\na v_j|^{p-2}\p_l{v_j}\p_m v_j \psi^l_m.
\]

The same compactness argument as in 
the proof of Lemma \ref{lem-dense} gives that 
$\|v_j\|_{W^{1,p}(B_{\frac34})}$ are uniformly bounded. 
Also, it implies (with the help of Proposition \ref{prop-seq-comp}) that we can extract 
a converging subsequence such that $v_j\to v_0$
uniformly in $\overline{B_{\frac34}}$ 
and $\na v_j\to \na v_0$ strongly
in $L^p(B_{\frac34})$. 
Moreover,  \eqref{dndooo-3}, \eqref{dndooo-1} and Theorem \ref{thm-entropy} give that
\begin{eqnarray*}
\Delta_p v_0(x)=0, \quad v_0(x)\ge 0\  \text{if}\ x\in B_{\frac34}, \quad  v_0(0)=0, 
\quad {\mbox{ and }} \sup_{B_{\frac12}}v_0=1
\end{eqnarray*}
which is in contradiction with the strong minimum principle. 
This shows \eqref{dndooo} and finishes the proof.

\end{proof}
\section{Lipschitz continuity of $u$: Proof of Theorem \ref{thm-A}}
\label{sec:main-theorem}
Proposition \ref{prop:lin flat} and Lemma \ref{lem-dense} can be summarized by saying that 
if at $x_0\in \fb u$ the free boundary is neither flat nor the set is $\{u<0\}$ is thick
then we have uniform linear growth at $x_0$.
Thus we only have to look at those free boundary points where $u<0$ is nontrivial, since in its complement we 
know that 
$u$ is Lipschitz.

We begin by introducing another notion of flatness,  suitable for the viscosity solutions,  
in terms of the $\e-$monotonicity of $u$. More precisely, we give the following 
definitions: 

\begin{defn}
We say that
$u\in C(B_1)$ is $\epsilon-$monotone in $B_{1-\epsilon}$
if there are a unit vector $e$ and an angle $\theta_0$ with 
$\theta_0 > \frac\pi 4$ (say) and $\epsilon >0$ (small)
such that, for every $\epsilon'\ge \epsilon $,
\begin{equation}\label{e-mon}
\sup_{B_{\epsilon' \sin\theta_0} (x)} u(y -\epsilon ' e) \le u(x).
\end{equation}
\end{defn}

We denote by~$\Gamma(\theta_0,e)$ the cone with axis~$e$ and opening~$\theta_0$. 

\begin{defn}\label{defn-flat}
Let $u$ be a viscosity solution in $B_1(x)$, with~$x\in \fb u$. 
We say that $u$ is $\epsilon-$monotone in the 
cone~$\Gamma(\theta_0,e)$ if it is~$\epsilon-$monotone 
in any direction~$\tau\in\Gamma(\theta_0,e)$. 

Furthermore, we say that $u$ is $\e$-monotone in
the cone $\Gamma(\theta_0, e)$ in $B_r(x)$ if 
the function~$U(y)=\frac{u(x+yr)}{r}$, with~$y\in B_1$, is so in the cylinder 
 $B'_{\frac1{\sqrt2}-\e} \times(-\frac1{\sqrt2}+\e, \frac1{\sqrt 2}-\e)\subset B_1$, where 
 $B'_r$ denotes the ball with radius $r$ of codimension 1.
\end{defn}

One can interpret the $\e-$monotonicity of $u$ as closeness of 
the free boundary to a Lipschitz graph with Lipschitz constant sufficiently close to 
$1$ if we leave the free boundary in directions $e$ at distance  $\e$ and larger. 
The exact value of the Lipschitz constant is given 
by~$\left(\tan\frac{\theta_0}{2}\right)^{-1}$. 
Then for suitable $\e$ and~$\theta_0$, which we call {\bf critical flatness constants},  the ellipticity propagates to 
the free boundary via Harnack's inequality giving that $\Gamma$ is Lipschitz. 
Furthermore, Lipschitz free boundaries are, in fact, $C^{1, \alpha}$ regular. 
Therefore we have
\medskip 

\begin{thm}\label{thm-lip}
Let $x_0\in \fb u$ such that $h(x_0, r)<r h_0$ and $\Theta(u, x_0, r)\ge \delta$
with  $\delta>0$ as in Lemma \ref{lem-dense}.
Then there is a constant $ C=C(n, M, \delta, h_0)$ such that 
\[
|u(x)|\leq C|x-x_0|, \quad x\in B_{\frac r2}(x_0). 
\]
\end{thm}
The proof is a slight modification of Theorem A \cite{DK}, since the condition 
$\Theta(u, x_0, r)\ge \delta$ implies that there is a negative phase and $u$
is a viscosity solution.

\section{Behaviour near free boundary}
\label{sec:basic}

With Lipschitz continuity we can show that the results in \cite{CLW-uniform}
hold for the nonlinear problem \eqref{pde-0}.  With minor modifications 
the following theorem follows from the results of \cite{CLW-uniform}.
\begin{thm}
Let $\uej$ be solutions to \eqref{pde-0} in a domain $D \subset \R^n$. Let
$x_0\in D$ and suppose $\uej$ converge to $u_0$ 
 uniformly on compact subsets of $D$ as $\e_j\to 0$. Then the following holds
 \begin{itemize}
 \item[(i)] if $u_0= \alpha(x-x_0)^+_1 -\gamma(x-x_0)^-_1$
with $\alpha\ge0,\gamma>0$  then
\[
 \alpha^p-\gamma^p =pM.
\]

\item[(ii)] if $u_0 = \alpha(x-x_0)^+_1$
 $\alpha\in \R$ then
\[
 0\le \alpha\le (pM)^{\frac1p}.
 \]

\item[(iii)] if $u = \alpha(x-x_0)^+_1 +\bar \alpha(x-x_0)^-_1$
 $\alpha>0,\bar \alpha>0$ then
\[
 \alpha=\bar\alpha \le (pM)^{\frac1p}.
 \]

 \end{itemize}

\end{thm}

\label{sec:fbb}
The next two theorems exhibit the behavior of $u$ near the free boundary
\begin{thm}
Let $\uej$ be solutions to \eqref{pde-0} in a domain $D\subset \R^n$ such that $\uej \to u$ 
uniformly on compact subsets of $D$ and $\e_j \to  0$. Let $x_0\in D \cap \fb u$ and let 
$\gamma \ge  0$ be such that
\[
\limsup_{ x\to x_0 } |\na u^-(x)| \le \gamma.
\]
Then,
\begin{equation}\label{vera-3}
\limsup_{ x\to x_0 } |\na u^+(x)| \le (pM+\gamma^p)^{\frac1p}.
\end{equation}
\end{thm}

\begin{proof}
We have divided the  proof into  six steps:

 Step 1) 
Let \[
\alpha:=\limsup_{\begin{subarray}
 \ x\to x_0\\ u(x)>0.
\end{subarray}
}
|\na u(x)|
\]  
By Theorem \ref{thm-lip} $u$ is Lipschitz continuous, therefore $\alpha$ is finite. 
If $\alpha =0$ then we are done. Thus let us assume that $\alpha>0$.
There is a sequence $x_k\in \po u$ such that $x_k\to x_0$ and $\lim_{k\to \infty}|\na u(x_k)|=\alpha$.
Denote $d_k=\dist(x_k, \fb u)$, then we know that there is $z_k\in \fb u$ such that 
$d_k=|x_k-z_k|$. 

Step 2) 
Let 
\[
u_{d_k}(x)=\frac1{d_k} u(z_k+d_k x).
\]
We have that $|\na u_{d_k}(x)|=|\na u(z_k+d_k x)|\le C$ because 
$u\in C^{0,1}_{loc}(D)$ by Theorem \ref{thm-lip}. Consequently, $u_{d_k}(x)$
are uniformly bounded on compact sets of $\R^n$ since $u_{d_k}(0)=0.$
Therefore there is a subsequence (still labelled $u_{d_k}(x)$) such that 
$u_{d_k}\to u_0$ uniformly on the compact subsets of $\R^n$ and 
the limit $u_0$ is Lipschitz continuous on the compact subsets of $\R^n$.

Step 3) 
Consider   $\bar x_k=\frac{x_k-z_k}{|x_k-z_k|}$ pointing 
into 
$\po {u_{d_k}}$. Note that $x_k\in \p B_1$ and $B_1(\bar x_k)\subset \po {u_{d_k}}$.
Se can exract a subsequence,  still labelled $\bar x_k$, such that 
$\bar x_k\to \bar x$ such that $u_0(x)\ge 0$ in $B_1(\bar x)$ and $\pl u_0=0$ in 
$B_1(\bar x)$.

We can also extract a converging subsequence from the sequence of 
unit vectors 
\[
\nu_k:=\frac{\na u_{d_k}(\bar x)}{|\na u_{d_k}(\bar x)|}
\]
still labelled $\nu_k$ such that $\nu_k\to \nu$.
We claim that  
\begin{equation}\label{vera}
|\na u(x_k)|\to \deriv{u_0}{\nu}(\bar x).
\end{equation}
Note that $\na u(x_k)=\na u_{d_k}(\bar x_k)$.
Hence it is enough to show check that 
\begin{equation}\label{grad-uniform}
\na u_{d_k}\to \na u_0 \quad \mbox{on compact subsets of} \ B_1(\bar x).  
\end{equation}
To see this we first note that $\psi(u_{d_k}-u_{d_m})\in W_{0}^{1, p}(B_1(\bar x))$
for given $0\le \psi \in \Cinf_0(B_1(\bar x))$ for sufficiently large $k, m$.
Therefore 
\begin{eqnarray}\nonumber
0
&=&
\int(|\na u_{d_k}|^{p-2}\na u_{d_k}-|\na u_{d_m}|^{p-2}\na u_{d_m})(\na (u_{d_k}-u_{d_m})\psi +(u_{d_k}-u_{d_m})\na \psi)\\\nonumber
&=&
 \int(|\na u_{d_k}|^{p-2}\na u_{d_k}-|\na u_{d_m}|^{p-2}\na u_{d_m})(\na (u_{d_k}-u_{d_m})\psi\\\nonumber
 &&+
\int(|\na u_{d_k}|^{p-2}\na u_{d_k}-|\na u_{d_m}|^{p-2}\na u_{d_m})\na \psi(u_{d_k}-u_{d_m})\\\nonumber
&\ge&
\gamma\int|\na u_{d_k}-\na u_{d_m}|^p\psi\\\label{punjab}
&&-\sup|\na \psi(u_{d_k}-u_{d_m})|\int|\na u_{d_k}|^{p-1}+|\na u_{d_m}|^{p-1}
\end{eqnarray}
where the last inequality follows from a well know estimate \eqref{eq:p-coerc} with $\gamma$
depending only on $n, p$. Thus for an appropriate choice of  $\psi\ge 0$ we get 
from \eqref{punjab} that 
\begin{equation}\label{p-grad}
\gamma\int_{B}|\na u_{d_k}-\na u_{d_m}|^p\le 2\|\na u_{d_k}\|_\infty^{p-1} \vol(2B)\sup_{2B}|u_{d_k}-u_{d_m}|
\end{equation}
for every ball $B$ satisfying $2B\Subset B_1(\bar x)$ for sufficiently large 
$k, m$. Here we assume that $2B$ is the ball with the same center as $B$ and of radius equal to 
the diameter of $B$.
On the other hand
\begin{eqnarray*}
|\na u_{d_k}(x)-\na u_{d_m}(x)|
&\le& 
\left| \na u_{d_k}(x)-\fint_{B_r(x)} \na u_{d_k}\right|+
\left|\fint_{B_r(x)} \na u_{d_k}-\fint_{B_r(x)} \na u_{d_m}\right|\\
&&+
\left|\fint_{B_r(x)} \na u_{d_m}- \na u_{d_m}(x)\right|\\
&\le&
2Cr^\beta + 2\|\na u_{d_k}\|_\infty^{p-1} \vol(B_{2r}(x))2^n\|u_{d_k}-u_{d_m}\|_{L^\infty(B_{2r}(x))}
\end{eqnarray*}
 where the last line follows from 
 the $\beta$-H\"older estimate  for gradient (see \cite{DM}) and \eqref{p-grad}. Since $r$ is arbitrary and 
 $\|u_{d_k}-u_{d_m}\|_{L^\infty(B_{2r}(x))}\to 0$, if $k, m$ are sufficiently large,  it follows that  
that $\na u_{d_k} \to \na u_0$ uniformly in some uniform neighborhood of  $\bar x$.
As result we get that 
\begin{equation}\label{vera-2}
\alpha\leftarrow |\na u(x_k)|=|\na u_{d_k}(\bar x_k)|=\pscal {\na u_{d_k}(\bar x_k)}{\nu_k}\rightarrow \deriv{u_0}\nu(\bar x)
\end{equation}
and \eqref{vera} follows.

Step 4) 
We claim that $|\na u^+_0|\le \alpha, |\na u^-_0|\le \gamma$ in $\R^n$. 
For every $\tau>0$ there is $\delta>0$ such that 
$\sup_{B_\tau(x_0)}|\na u^+|<\alpha+\delta.$ For fix $R>0$, 
$\na u_{d_k}(x)=\na u(z_k+d_kx)|<\alpha+\delta$ if 
$d_k$ is sufficiently small so that 
$|x_0-(z_k+d_kx)|\le |x_0-z_k|+d_k R=(1+R)d_k<\tau$. 
Thus  $\sup_{B_R}|\na u_{d_k}|\le \alpha+\delta$ and hence $\sup_{B_R}|\na u_0|\le \alpha+\delta$.
Since $\delta>0$ is arbitrary the claim follows. By a similar argument we can prove that 
$|\na u^-|\le\gamma $.

Step 5) 
Let $v=\deriv {u_0}\nu$. Then differentiating $\pl u_0=0$ in $\nu$ direction we get that 
$\div(a(\na u_0)\na v)=0$ in $B_1(\bar x)$, where 
$a(\na u_0)$ is a matrix with $p$-laplacian type growth. Since by \eqref{vera-2}
$\na u_0\not =0$ near $\bar x$, it follows that $v$ solves a uniformly elliptic 
equation  in $B_R(\bar x)$ for some $R>0$ small. 
Since $v$ attains local maximum at $\bar x$ then it follows that 
$v=\alpha$ in $B_R(\bar x)$. For the sake of simplicity we assume that 
$\nu=e_1$, thus $u=\alpha x_1+g(x'), x'=(0, x_2, \dots, x_n)$
for some function $g$. Form $|\na u_0|\le \alpha$ it follows that $g$ must be constant.
From the unique continuation theorem \cite{Gran} it readily follows that 
there is a point $\tilde x$ such that 
\[
u_0(x)=\alpha(x-\tilde x)_1^+\quad \mbox{in}\ (x-\tilde x)_1>0, 
\]
and 
\[
|\na u_0^-|\le \gamma\quad \mbox{in}\ \R^n. 
\]
On the other hand from the asymptotic expansion \cite{DK}
we have that there are $\bar \alpha, \bar \gamma$ such that 
\begin{eqnarray*}
u_0^+(x)=\bar \alpha(x-\tilde x)_1^-+o(|x-\tilde x|), \quad \mbox{in}\ (x-\tilde x)_1<0\\
u_0^-(x)=\bar \gamma(x-\tilde x)_1^-+o(|x-\tilde x|), \quad \mbox{in}\ (x-\tilde x)_1<0
\end{eqnarray*}
Step 6)  
To finish the proof we blow-up $u_0$ one more time. Let $u_{0\lambda}(x)=\frac1{\lambda}u_0(\tilde x+\lambda x)$. From Step 5 we conclude that for a subsequence these functions converge to 
$u_{00}=\alpha x_1+\mu x_1^-$. From Proposition \ref{prop-2nd-blow} it follows that 
there is a sequence $\e_j^{00}$ such that $u^{\e_j^{00}}$ are solutions to 
\eqref{pde-0} and $u^{\e_j^{00}}\to u_{00}=\alpha x_1+\mu x_1^-$.
If $\mu=0$ then Theorem 8 (ii) gives \eqref{vera-3}. If $\mu>0$ then 
from Theorem 8 (iii). If $\mu<0$ then since $\na u_{0\lambda_k}\to \na u_{00}$
$\ast$-weakly in $L^\infty_{loc}$ and $|\na u^-|\le \gamma$ it follows that 
$|\mu|\le \gamma$ and we can apply Theorem 8 (i). 
\end{proof}

\begin{thm}
Let $\uej$ be a solution to $P\e_j$ in a domain 
$D_j \subset \R^n$ such that $D_j \subset D_{j+1}$ and   $\cup_j D_j = \R^n$. 
Let us assume that $\uej$ converges
to a function $U$ uniformly on compact sets of $\R^n$ and $\e_j \to 0$. 
Assume, in
addition,that $U\in Lip(1,1)$ in $\R^n$ and $\fb U=\emptyset$. 
If $\gamma \ge 0$ is such that 
$|\na U^-| \le \gamma$ in $\R^n$ 
then,
\[
|\na U^+| \le  \sqrt{2M +\gamma^2} \quad \mbox{in}\ \ \R^n.
\]
\end{thm}

Proof follows from minor modifications of the previous one.

\section{Application: Weak solutions}\label{sec:weak}
In this section we study the set of singular points of \textit{weak solutions}, a subclass of variational solutions. The main result of this section states that 
the weak energy identity also holds for the weak solutions, hence,
if \eqref{BMO} holds, one can prove their local Lipschitz regularity, as in Theorem \ref{thm-A}.
We begin with the following definition of the weak solutions, \cite{AC}, \cite{Weiss} .
\begin{defn}\label{weak-def}
A function  $u$ is said to be a weak solution of our free boundary problem
if the following is satisfied:
\begin{itemize}
 \item [\rm{1)}] $u\in W^{1, p}(\Omega)$
is continuous  in $\Omega$ and $p$-harmonic in $(\{u>0\}\cup \{u<0\})\cap \Omega$,
\item [\rm{2)}] for $D\Subset\Omega $, $\{u>0\}\cap D$ is a set of finite perimeter,  and 
\begin{itemize}
\item[$\bf1^\circ$] $\partial_{\rm red}\{u>0\}$ is open relative $\partial\{u>0\}$,
\item[$\bf2^\circ$] $\partial_{\rm red}\{u>0\}$
is smooth,
\item[$\bf3^\circ$] $\H^{n-1}(\partial\{u>0\}\setminus \partial_{\rm red}\{u>0\})=0$.
\end{itemize}
$\partial_{\rm red}\{u>0\}$ is the reduced boundary of $\{u>0\}$, see 4.5.5. \cite{Federer} for definition.
\item[\rm{3)}]  
On $\p_{\rm red }\{u>\}$ we have the free boundary condition satisfied 
\[
(p-1) (|\na u^+|^p-|\na u^-|^p)=\lambda^p. 
\]
\end{itemize}

\end{defn}

\begin{lem}
Let $u$ be a weak solution in the sense of Definition \ref{weak-def}.
Then
\[
\int(|\nabla
u|^p+\lambda^p\I{u})\mbox{div}\varphi-p|\nabla
u|^{p-2}(\nabla uD\varphi)\cdot\nabla u
=0, \quad \forall \varphi\in C_0^1(\Omega, \R^n).
\]
\end{lem}

\begin{proof}
For a given test function $\phi\in C^{0,1}(\Omega;\R^n)$ the set $\supp\phi\cap(\partial\{u>0\}\setminus \partial_{\rm red}\{u>0\})$ is compact.
Consider a covering of this set by $B_{r_i}(x_i)$ such that
$$
\supp\phi\cap(\partial\{u>0\}\setminus \partial_{\rm red}\{u>0\})\subset \cup_{i=1}^\infty B_{r_i}(x_i)
$$
and $\sum_{i=1}^\infty r^{n-1}_i <\delta$. 
Then there is a finite subcovering 
 $\mathscr F=\displaystyle\cup_{i=1}^{N(\delta)}B_{r_i}(x_i)$
such that $\supp \phi\cap\left(\partial \{u>0\}\setminus \partial_{\rm{red}}\{u>0\}\right)\subset\mathscr F$
and $\sum_{i=1}^{N(\delta)}r^{n-1}_i<\delta$.
Splitting the integral into two parts, we have to show that
$$
\int_{\supp \phi}(|\nabla
u|^p+\lambda^p\I{u})\mbox{div}\varphi-p|\nabla
u|^{p-2}(\nabla uD\varphi)\cdot\nabla u
=
\int_{\mathscr F}+\int_{\supp\phi\setminus\mathscr F}=0.
$$

By 2) Definition \ref{weak-def} $\{u>0\}\cap (\supp\phi\setminus \mathscr F)$ is of finite perimeter.
By \cite{Giusti} Theorem 1.24 and Remark 1.27 there are sets
$E_j^+\subset \{u>0\}$ with $C^\infty$ boundaries which approximate $E=(\supp\phi\setminus \mathscr F)\cap\{u>0\}$ from inside.
After partial integration we obtain
\begin{eqnarray}
 \sum_{l}\int_{E_j^+}|\na u|^p\phi^l_l-p\sum_{lm}\int_{E_j^+} |\na u|^{p-2}u_m\phi^l_mu_l= \int_{\partial E_j^+}|\na u|^{p}\phi\cdot\nu d\H^{n-1}\\\nonumber
- p\sum_{lm}\int_{E_j^+}|\na u|^{p-2}u_mu_{ml}\phi^l+|\na u|^{p-2}u_m\phi^l_mu_l\\\nonumber
=\int_{\partial E_j^+}|\na u|^{p}\phi\cdot\nu d\H^{n-1}- p\sum_{lm}\int_{E_j^+}|\na u|^{p-2}u_m[u_{l}\phi^l]_m\\\nonumber
=\int_{\partial E_j^+}\left\{|\na u|^{p}\phi\cdot\nu- p\sum_{lm}|\na u|^{p-2}u_mu_{l}\phi^l\nu_m\right\} d\H^{n-1},\\\nonumber
\end{eqnarray}
where to get the last line we used $\Delta_p u=0$ in $E_j^+$. Note that 
the integrals in above computation involving the second order derivatives of $u$ are well  defined thanks to the 
weighted local $W^{2,2}$ estimates for $p-$harmonic functions.
Thus
\begin{eqnarray}
 \int_{E_j^+}\left[|\na u|^p+\lambda^p\I{u}\right]\div \phi-p|\na u|^{p-2}\na u D\phi\na u=\\\nonumber
=\int_{\partial E_j^+}\left\{\left[|\na u|^{p}+\lambda^p\right]\phi\cdot\nu- p\sum_{ml}|\na u|^{p-2}u_mu_{l}\phi^l\nu_m\right\} d\H^{n-1}.
\end{eqnarray}

Using a similar approximation argument with $E_j^-\subset \{u<0\}$ 
we infer
\begin{eqnarray}
 \int_{E_j^-}|\na u|^p\div \phi-p|\na u|^{p-2}\na u D\phi\na u=\\\nonumber
=\int_{\partial E_j^-}\left\{|\na u|^{p}\phi\cdot\nu- p\sum_{ml}|\na u|^{p-2}u_mu_{l}\phi^l\nu_m\right\} d\H^{n-1}.
\end{eqnarray}
Since 
$\partial_{}\{u>0\}\setminus \mathscr F\subset \partial_{\rm red}\{u>0\}$,
on $\partial_{\rm red}\{u>0\}$ we have
$u_m=-|\na u|\nu_m$ and the free boundary condition  $(p-1)(|\na u^+|^{p}-|\na u^-|^p)=\lambda^p$ is satisfied, we 
conclude
\begin{eqnarray*}
\lim _{j\to \infty}\int_{(E^+_j\cup E^-_j)\setminus \mathscr F}
\left[|\na u|^p+\lambda^p\I{u}\right]\div \phi-p|\na u|^{p-2}\na u D\phi\na u\\
=
\int_{\partial\{u>0\}\setminus \mathscr F}
\left[\lambda^p-(p-1)(|\na u^+|^{p}-|\na u^-|^p)\right]\phi\cdot\nu d\H^{n-1}=0.
\end{eqnarray*}

Thus the integral over $\supp\phi\setminus \mathscr F$ is 0. The remaining integral 
$$
\left|\int_{\mathscr F}\left[|\na u|^p+\lambda^p\I{u}\right]\div \phi-p|\na u|^{p-2}\na u D\phi\na u\right|
\le 
(p+1)\|D \varphi \|_{\infty}\int_{\mathscr F}|\na u|^p
$$ 
tends to zero  as $\delta\to 0$ since 
$\sum_{i=1}^{N(\delta)}r^{n}_i<\delta^2$ and we can utilize the absolute continuity of the integral.  
Sending $\delta $ to zero the result follows.
\end{proof}

\section{A BMO estimate}\label{sec:BMO}
\label{sec:bmo}
In \cite{DK} we have proven that the gradient of a minimizer is locally BMO provided that $p>2$. 
A weaker estimate holds for the solutions of \eqref{pde-0}.
More precisely, in this section we show 
that the tensor $T_{lm}(\na \ue)=p|\na \ue|^{p-2}\ue_l\ue_m-|\na \ue|^p\delta_{lm} $
can be decomposed to a sum of divergence free and BMO tensors.  
\begin{thm}
Let $\ue$ be a family of solutions to \eqref{pde-0}. Then the 
tensor $T_{lm}(\na \ue)=p|\na \ue|^{p-2}\ue_l\ue_m-|\na \ue|^p\delta_{lm} $
admits the following decomposition
\[
T_{ij}=T^{0}_{ij}+\widehat {T}_{ij}, 
\]
where $\p_j(T^0_{ij}(\na \ue))=0$ and $\widehat T_{ij}(\na \ue)\in BMO_{loc}(B_1)$, uniformly in $\e$,  
such that $\sup_\e\|\widehat T\|_{BMO(\mathcal C)}<\infty$ for every 
compact $\mathcal C\subset B_1$.   
\end{thm}

\begin{proof}
We recall Bogovski's formula \cite{Galdi} III 3.9. :
for every $ \omega \in C^{\infty }_{0}\left( \mathbb{R} ^{n}\right)$ with 
$ \supp\omega \subset B_{1}\left( 0\right)$ and 
$ \int _{B_{1}}\omega =1$
consider the vectorfield 

\[
v\left( x\right) =\int _{\Omega }f\left( y\right) \dfrac {x-y}{\left| x-y\right| ^{n}}\left[ \int ^{+\infty }_{1}\omega \left( y+r\left( x-y\right) \right) r^{n-1}dr\right] dy,
\]
where  $\Omega$ is a bounded domain
and   $f\in L^{q}\left( \Omega \right), q>1$ such that  $\int _{\Omega }f=0$. 
Then 
  $v\in W^{1,q}$ and 
\[
 \left\| v\right\| _{1,q}\leq C\left\| f\right\| _{q}
\]
Furthermore,  
\[
\div v=f\quad \mbox{in} \ \Omega.
\]

Note that in Bogovski's formula  $v$ has the form $v(x)=\int_\Om k(x, y)f(y)$ with a singular kernel 
$k(x, y)$, and  the derivatives of $k$ behave like 
Calderon-Zygmund kernels \cite{Galdi} III 3.15-3.17.

Let $\eta$ be a cut off function of some ball $B\Subset B_1$. 
Localizing $ T_{ij}$ we have that $\p_{j}(\eta T_{ij})=f^i$
where 
\[
f^{i}=\eta \partial _{i}\mathcal B^\ast+\eta _{i}T_{ij},    
\]
with $\int f^{j}=0$. 
Hence from Bogovski's formula and the estimates for the Calderon-Zygmund operators, 
we get that there is $\widehat {T}_{ij}\in BMO(B)$
such that 
\[
T_{ij}=T^{0}_{ij}+\widehat {T}_{ij}, 
\]
where $\p_j(T^0_{ij})=0$. 

\end{proof}
\begin{rem}
Using the techniques from \cite{DK}, \cite{MR3390082} one can show that if $\ue $ are minimizers of 
$\int_{B_1}|\na v |^p+p\mathcal B(v/\e)$ then $T^0\in BMO_{loc}$ locally uniformly  for every $p>1$. 
Note that $\mbox{Trace}T_{lm}=(p-n)|\na \ue |^p$, thus if $T^0\in BMO_{loc}(B_1)$ then 
$BMO$ estimate  translate to $T$ provided that $p\not=n.$ 
\end{rem}

\end{document}